\documentclass[english,11pt,a4paper,leqno]{amsart}
\usepackage{amsmath,amstext, amsthm, amssymb}
\usepackage{latexsym, amsfonts, graphicx, xcolor}
\usepackage[latin1]{inputenc} 
\usepackage[T1]{fontenc} 
\usepackage{enumerate}
\usepackage[english]{babel} 
\usepackage{enumitem}
\usepackage{hyperref}

\makeatletter
\def\namedlabel#1#2{\begingroup
    #2%
    \def\@currentlabel{#2}%
    \phantomsection\label{#1}\endgroup
}

\makeatother

\hypersetup{
  colorlinks   = true, 
  urlcolor     = blue, 
  linkcolor    = blue, 
  citecolor   = red 
}

\numberwithin{equation}{section} 

\newtheorem{thm}{Theorem}[section]

\newtheorem{lem}[thm]{Lemma}
\newtheorem{lem-app}{Lemma}
\newtheorem*{lem*}{Lemma}

\newtheorem*{prop*}{Proposition}

\theoremstyle{remark}
\newtheorem{rem}[thm]{Remark}

\theoremstyle{definition} 
\newtheorem{defi}[thm]{Definition}

\newtheorem*{nota*}{Notation}

\newcommand{\Q}{\overline{\mathbb Q}}

\newcommand{\C}{\mathbb{C}}
\newcommand{\Z}{\mathbb{Z}}

\newcommand{\I}{\mathcal I}

\newcommand{\J}{\mathcal J}

\newcommand{\lambd}{{\boldsymbol{\lambda}}}
\newcommand{\f}{{\boldsymbol f}}

 \renewcommand{\k}{{\boldsymbol{k}}}

\newcommand{\X}{\boldsymbol{X}}

\newcommand{\Y}{\boldsymbol{Y}}

\newcommand{\bmu}{{\boldsymbol{\mu}}}
\newcommand{\bnu}{{\boldsymbol{\nu}}}

\newcommand{\bphi}{{\boldsymbol{\phi}}}

\newcommand{\btau}{{\boldsymbol{\tau}}}

\newcommand{\bTheta}{{\boldsymbol{\Theta}}}

\title[A new proof of Nishioka's theorem]{A  new proof of Nishioka's theorem in Mahler's method}

\author{Boris Adamczewski}
\address{
Univ Lyon, Universit\'e Claude Bernard Lyon 1\\
 CNRS UMR 5208, Institut Camille Jordan \\
 F-69622 Villeurbanne Cedex, France}
\email{Boris.Adamczewski@math.cnrs.fr}

\author{Colin Faverjon}
\address{
Univ Lyon, Universit\'e Claude Bernard Lyon 1\\
 CNRS UMR 5208, Institut Camille Jordan \\
 F-69622 Villeurbanne Cedex, France}
\email{faverjon@math.univ-lyon1.fr}
\date{}

\thanks{ }
\begin{document}

\begin{abstract}
In a recent work \cite{AF20}, the authors established new results about general linear Mahler systems in several variables 
from the perspective of transcendental number theory, such as a multivariate extension of Nishioka's theorem. 
Working with functions of several variables and with different Mahler transformations leads to a number of complications, including the need to prove a general vanishing theorem and to use tools from ergodic Ramsey theory and Diophantine approximation (\emph{e.g.}, 
a variant of the $p$-adic Schmidt subspace theorem).   
These complications make the proof of the main results proved in \cite{AF20} rather intricate.  
In this article,  we describe our new approach in the special case of linear Mahler systems 
in one variable. This leads to a new, elementary, and self-contained proof of Nishioka's theorem, as well as of the lifting theorem more recently obtained by Philippon \cite{PPH} and the authors \cite{AF17}. Though the general strategy remains the same as in \cite{AF20}, the proof turns out to be greatly simplified. Beyond its own interest, we hope that reading this article will facilitate the understanding of the proof of the main results obtained in \cite{AF20}. 
\end{abstract}
\bibliographystyle{abbvr}
\maketitle

\section{Introduction}

Throughout this paper, we let  $q\geq 2$ denote a fixed integer. An \emph{$M_q$-function} is a power series $f(z)\in\Q[[z]]$ satisfying 
a linear equation of the form
$$
p_0(z)f(z) + p_1(z)f(z^q)+\cdots +p_m(z)f(z^{q^m})=0\,,
$$
where $p_0(z),\ldots,p_m(z)\in\Q[z]$ are not all zero. 
In the study of $M_q$-functions, it is often more convenient to consider, instead of linear Mahler equations, linear systems  of functional equations of the form   
\begin{equation}\label{eq:P2-MahlerSystem}
\left(\begin{array}{c}f_1(z) \\ \vdots \\ f_m(z)\end{array}\right)
=
A(z)
\left(\begin{array}{c}f_1(z^q) \\ \vdots \\ f_m(z^q)\end{array}\right)\,, 
\end{equation}
where  $A(z) \in {\rm GL}_m(\Q(z))$ and $f_1(z),\ldots,f_m(z)\in\Q[[z]]$. Then, each power series $f_i(z)$ is an $M_q$-function.  We recall that an $M_q$-function 
is meromorphic in the open unit disc of $\mathbb C$ (see, for instance, \cite[Th\'eor\`eme 31]{Du}). Furthermore, it admits the unit circle as a natural boundary, unless it is a rational function \cite[Th\'eor\`eme 4.3]{Ra92}. 
A point $\alpha\in\mathbb C$ is said to be 
\emph{regular} with respect to \eqref{eq:P2-MahlerSystem} if the matrix $A(\alpha^{q^k})$ is both well-defined and invertible 
for all integers $k\geq 0$. 

In this framework, the main aim of Mahler's method is to transfer results about the absence of algebraic
(resp.\ linear) relations between the functions $f_1(z),\ldots,f_m(z)$ over $\Q(z)$ to the absence
of algebraic (resp.\ linear) relations over $\Q$ between their values at non-zero algebraic points lying in the open unit disc (assuming, of course, that these values are well-defined). 
In 1990, Ku.\ Nishioka \cite{Ni90} proved the following theorem, which is the analog of the Siegel-Shidlovskii theorem 
in the theory of Siegel $E$-functions (see \cite{Sh_Liv}).  
Given a field $\mathbb K$, a field extension $\mathbb L$ of $\mathbb K$, and elements 
$a_1,\ldots,a_m$ in $\mathbb L$, 
we let ${\rm tr.deg}_{\mathbb K}(a_1,\ldots,a_m)$ denote the transcendence degree over 
$\mathbb K$ of the field extension $\mathbb K(a_1,\ldots,a_m)$. 

\begin{thm}[Nishioka's theorem]\label{thm:P2-Nishioka}
Let $f_1(z),\ldots,f_m(z)$ be $M_q$-functions related by a $q$-Mahler system of the form \eqref{eq:P2-MahlerSystem} and let $\alpha \in \Q$, 
$0< \vert \alpha \vert < 1$, be regular with respect to this system. Then
$$
{\rm tr.deg}_{\Q}(f_1(\alpha),\ldots,f_m(\alpha)) = {\rm tr.deg}_{\Q(z)}(f_1(z),\ldots,f_m(z))\, .
$$
\end{thm}

Nishioka's theorem is undoubtedly a landmark result in Mahler's method, but it also suffers from some limitation which prevent it to cover important applications (see the discussion in Sections 1 and 2 of \cite{AF17} and also the results in \cite{AF18}).  For such applications, the following refinement of Nishioka's theorem, which we called \emph{lifting theorem} (or \emph{th\'eor\`eme de permanence} in French), is needed.

\begin{thm}[Lifting theorem]\label{thm:P2-Philippon}
Let $f_1(z),\ldots,f_m(z)$ be $M_q$-functions related by a $q$-Mahler system of the form \eqref{eq:P2-MahlerSystem} and let $\alpha \in \Q$, 
$0< \vert \alpha \vert < 1$, be regular with respect to this system. Then for any homogenous polynomial $P \in \Q[X_1,\ldots,X_m]$ such that
$$
P(f_{1}(\alpha),\ldots,f_{m}(\alpha)) = 0\,,
$$
there exists a polynomial $\overline{P}\in \Q[z,X_1,\ldots,X_m]$, homogeneous in $X_1,\ldots,X_m$, such that 
\begin{eqnarray*}
\overline{P}(z,f_{1}(z),\ldots,f_{m}(z)) = 0 & \mbox{ and } &
\overline{P}(\alpha,X_1,\ldots,X_m)=P(X_1,\ldots,X_m).
\end{eqnarray*}
\end{thm}

Again, Theorem~\ref{thm:P2-Philippon} has an analog in the theory of $E$-functions: the lifting theorem proved by Beukers \cite{Be06} using Andr\'e's theory of arithmetic Gevrey series \cite{An1,An2}.  
A slightly weaker version of Theorem~\ref{thm:P2-Philippon} was first proved by Philippon \cite{PPH}.  Theorem~\ref{thm:P2-Philippon} was then deduced in  \cite{AF17} 
from Philippon's lifting theorem.  In \cite{AF17,PPH}, the lifting theorem is derived from Nishioka's theorem. Thanks to the work of Andr\'e \cite{An3}, pursued by Naguy and Szamuely \cite{NS20}, we now have a general 
approach based on a suitable Galois theory of linear differential and difference equations that allows one to deduce theorems of the type of Theorem~\ref{thm:P2-Philippon} from theorems of the type of 
Theorem~\ref{thm:P2-Nishioka}. 

The proof of Nishioka's theorem deeply relies on tools from commutative algebra, related to elimination theory, which were introduced and developed by Nesterenko  in the framework of transcendental number theory at the end of the 1970s (see, for instance, \cite{Ne77}). Recently, Fernandes \cite{Fe18} observed that Nishioka's theorem can also be  derived from a general 
algebraic independence criterion due to Philippon \cite{PPH1}. However, Philippon's criterion is also based on the same tools, so that, in the end, both proofs rely on the same argument. The proof of Nishioka's theorem has the advantage that it can be quantified (see, for instance, \cite{Ni90}), 
leading to algebraic independence measures. Its main deficiency is that it can hardly be generalised to Mahler systems in several variables.

In this note, we use the approach recently introduced by the authors \cite{AF20} to provide new and more elementary proofs of both Nishioka's theorem and the lifting theorem. 
This approach takes its roots in the original one initiated by Mahler \cite{Ma30b} and developed much later by Kubota \cite{Ku77}, Loxton and van der Poorten \cite{LvdP82}, and Nishioka \cite{Ni94,Ni96}. The main improvement comes from the introduction of the so-called \emph{relation matrices} whose existence is ensured by Hilbert Nullstellensatz. 
In contrast with \cite{AF17,PPH}, we first prove the lifting theorem and then deduce Nishioka's theorem by using a classical 
argument, as in Shidlovskii's proof of the Siegel-Shidlovskii theorem (see \cite{Sh_Liv} or \cite{FN}). Beyond its elementary aspect, this new approach has the great advantage of being generalisable within the framework of Mahler's method in several variables, as has been done in \cite{AF20}. We hope that reading first 
this article will facilitate the understanding of the proof of the main results in \cite{AF20}. 
In order to avoid the proofs of  Theorems \ref{thm:P2-Nishioka} and \ref{thm:P2-Philippon}  being buried in tedious computations, we occasionally just outline the main argument and provide more detail in the appendix at the end of the paper. 
There, we also prove some auxiliary results 
that can be used to make our proofs of Theorems \ref{thm:P2-Nishioka} and \ref{thm:P2-Philippon} as elementary and self-contained as possible.


\section{Lifting the linear relations}\label{sec:P2-LiftingLinear}

We first prove Theorem \ref{thm:P2-Philippon} in the particular case of linear relations.

\begin{thm}\label{thm:P2-liftinglinear}
Let $f_1(z),\ldots,f_m(z)$ be $M_q$-functions related by a $q$-Mahler system of the form \eqref{eq:P2-MahlerSystem}, and let $\alpha \in \Q$, $0 < \vert \alpha \vert < 1$, be regular with respect to this system. Let $L \in \Q[X_1,\ldots,X_m]$ be a linear form such that
$$
L(f_{1}(\alpha),\ldots,f_{m}(\alpha)) = 0\, .
$$
Then, there exists $\overline{L}\in \Q[z,X_1,\ldots,X_m]$, linear in $X_1,\ldots,X_m$, such that 
\begin{eqnarray*}
\overline{L}(z,f_{1}(z),\ldots,f_{m}(z)) = 0 & \mbox{ and } &
\overline{L}(\alpha,X_1,\ldots,X_m)=L(X_1,\ldots,X_m)\,.
\end{eqnarray*}
\end{thm}

The proof of this theorem is dividing in three subsections. We first establish the existence and properties of some special matrices 
which can be associated with a linear Mahler system. We call them the \textit{relation matrices}. Then, we construct an auxiliary function and use it to prove a \textit{key lemma} about the structure of the linear relations between $f_1(z),\ldots,f_m(z)$.  Finally, we show how this lemma allows us to lift any linear relation over $\Q$ between $f_{1}(\alpha),\ldots,f_{m}(\alpha)$ into a linear relation over $\Q(z)$ between $f_1(z),\ldots,f_m(z)$. 
Throughout  this section, we keep the notation of Theorem~\ref{thm:P2-liftinglinear}. 
\subsection{Notation}\label{sec: notation}

Let $d$ be a positive integer and $R$ be a commutative ring. 
Given an indeterminate $x$, we let $R[[x]]$ denote the ring of formal power series with coefficients in $R$. If $R\subset \mathbb C$, 
we let $R\{x\}$ denote the ring of convergent power series with coefficients in $R$, that is those elements of $R[[x]]$ that are analytic in some neighborhood of the origin. 
Given a $d$-tuple of non-negative integers 
$\k=(k_1,\ldots,k_d)$, we set $\vert \k\vert :=k_1+\cdots +k_d$. 
If $X_1,\ldots,X_d$ are indeterminates, we set $\X^\k := X_1^{k_1}\cdots X_d^{k_d}$. 
The total degree of a polynomial in $R[X_1,\ldots,X_d]$  is defined by 
$$\deg \left(\sum_{\k\in K} a_{\k}\X^{\k}\right):= \max \{\vert \k\vert : \k \in K\,, \, a_{\k}\not=0\}\,.$$

Given an $m\times n$ matrix $M:=(m_{i,j})$  with coefficients in $R$ and 
an $m\times n$ matrix $\bmu=(\mu_{i,j})$  with nonnegative integer coefficients, we set 
$$
M^\bmu := \prod_{\substack{1\leq i\leq m\\1\leq j\leq n }} m_{i,j}^{\mu_{i,j}}\, .
$$

We use the standard Landau notation $\mathcal O$. We also use the notation $\gg$ as follows.  
Writing that some property holds true for all integers $\lambda\gg1$ means that the corresponding property holds true for all $\lambda$ large enough; 
writing that some property holds true for all integers 
$\lambda_1\gg \lambda_2,\lambda_3$ means 
that the corresponding property holds true for all $\lambda_1$ that is sufficiently large w.r.t.\ 
$\lambda_2$ and $\lambda_3$; writing that some property holds true for all integers 
$\lambda_1\gg \lambda_2\gg \lambda_3$ means 
that the corresponding property holds true for all $\lambda_1$ that is sufficiently large w.r.t.\ $\lambda_2$, 
assuming that $\lambda_2$ is itself sufficiently large w.r.t.\ $\lambda_3$.

\subsection{Relation matrices}\label{sec:P2-matrix}
 
To shorten the notation, we set $$\f(z):=(f_1(z),\ldots,f_m(z))^\top\,.$$ 
For every integer $k\geq 0$, we set 
$$
A_k(z) :=A(z)A(z^q)\cdots A(z^{q^{k-1}})\,,
$$
 so that $A_0(z)={\rm I}_m$, the identity matrix of size $m$, $A_1(z)=A(z)$, and 
\begin{equation}
\label{eq:P2-Mahleritere}
\f(z)=A_k(z)\f(z^{q^k})\,,\quad\quad \forall k \geq 0\,.
\end{equation}
Let $\Y:=(y_{i,j})_{1\leq i,j\leq m}$ denote a matrix of indeterminates. Given a field $\mathbb K$ and a non-negative integer $\delta_1$, we let $\mathbb K[\Y]_{\delta_1}$ 
denote the set of polynomials of degree at most $\delta_1$ in each indeterminate  
$y_{i,j}$.  Given two non-negative integers $\delta_1$ and $\delta_2$, we let 
 $\mathbb K[\Y,z]_{\delta_1,\delta_2}$ denote the set of polynomials $P\in \mathbb K[\Y,z]$ of degree at most $\delta_1$ in every indeterminate $y_{i,j}$ and of degree at most $\delta_2$ in $z$.  
The identity theorem and the fact that $\alpha$ is a regular point with respect to \eqref{eq:P2-MahlerSystem} ensure that 
every polynomial $P\in \Q(z)[\Y]$ is well-defined at the point $(A_k(\alpha),\alpha^{q^k})$ for all $k\gg 1$.
Set  
\begin{equation*}
\I := \{P \in \Q(z)[\Y] \ : \ P(A_k(\alpha), \alpha^{q^k})= 0\, , 
\; \forall k\gg 1 \} \, .
\end{equation*}

\subsubsection{Estimates for the dimension of certain vector spaces}

Let $\delta_1$ and $\delta_2$ be two non-negative integers. Set 
$\mathcal I(\delta_1) := \mathcal I \cap  \Q(z)[\Y]_{\delta_1}$ and 
$\I(\delta_1,\delta_2):=\I \cap \Q[\Y,z]_{\delta_1,\delta_2}$. 
Note that  $\I(\delta_1,\delta_2)$ is a vector subspace of $ \Q[\Y,z]_{\delta_1,\delta_2}$, and   
let $\I^{\perp}(\delta_1,\delta_2)$ denote a complement to $\I(\delta_1,\delta_2)$ in 
$ \Q[\Y,z]_{\delta_1,\delta_2}$.

\begin{lem}\label{lem:P2-dimensionespace} Let $d(\delta_1,\delta_2)$ denote the dimension of  
$\I^\perp(\delta_1,\delta_2)$ over $\Q$. 
There exists a positive integer $c_1(\delta_1)$, that does not depend on $\delta_2$, such that 
$$
d(\delta_1,\delta_2) \sim c_1(\delta_1)\delta_2 \,, \mbox{ as $\delta_2$ tends to infinity. }
$$
\end{lem}

\begin{proof}
Set $h:=(\delta_1+1)^{m^2}$ and let $\bnu_1,\ldots,\bnu_h$ denote an enumeration of the set of all  
matrices in ${\mathcal M}_m(\Z_{\geq 0})$ whose entries are at most $\delta_1$. Any polynomial $P \in \Q(z)[\Y]_{\delta_1}$ 
has a unique decomposition of the form    
$$
P(\Y,z):= \sum_{j=1}^h p_{j}(z) \Y^{\bnu_j}\, ,
$$
where $p_{j}(z) \in \Q(z)$, $1 \leq j \leq h$. 
Since, by definition, $\I(\delta_1)$ does not contain any non-zero elements of $\Q$, it is a strict $\Q(z)$-subspace of 
$\Q(z)[\Y]_{\delta_1}$. 
Thus, there exist an integer $d\geq 1$ and $d$ vectors of polynomials $(b_{i,1}(z),\ldots,b_{i,h}) \in \Q[z]^h$, $1\leq i \leq d$, 
which are linearly independent over $\Q(z)$ and such that  
for all $p_1(z),\ldots,p_h(z) \in \Q(z)$:
\begin{equation}\label{eq:P2-conditionI}
\sum_{j=1}^h p_j(z)\Y^{\bnu_j} \in \I(\delta_1) \Leftrightarrow \sum_{j=1}^h b_{i,j}(z)p_j(z) = 0 
\;\;\;\; \forall i,\; 1 \leq i \leq d\, .
\end{equation}
Since these polynomials only depend on $\delta_1$ (and $\mathcal I$), there exists $\delta'_1\geq 0$, which only depends on $\delta_1$ 
(and $\mathcal I$), such that 
$$
 b_{i,j}(z)=:\sum_{\kappa =0}^{\delta_1'} b_{i,j,\kappa} z^\kappa\, , \quad\quad b_{i,j,\kappa}\in\Q\,.
$$ 
Let us consider  
$P(\Y,z)=\sum_{j=1}^h p_j(z)\Y^{\bnu_j} \in \Q[\Y,z]_{\delta_1,\delta_2}$ and set  
$$
p_j(z) =: \sum_{\lambda \in \Z} p_{j,\lambda} z^\lambda \, ,
$$
where the numbers $p_{j,\lambda}$ belong to $\Q$ and $p_{j,\lambda}:=0$ if $\lambda >\delta_2$ or $\lambda < 0$.  
By  \eqref{eq:P2-conditionI},  $P$ belongs to $ \I(\delta_1,\delta_2)$ if and only if 
\begin{equation}\label{eq: gammarange}
\sum_{j=1}^h \sum_{ \kappa=0}^{\delta_1'} b_{i,j,\kappa}  p_{j,\gamma-\kappa} = 0 \, , \quad\quad \forall (\gamma,i)\,, \, 0\leq \gamma \leq  \delta_2+\delta'_1\,,\, 1\leq i\leq d\,.
\end{equation}
The number of linearly independent equations in \eqref{eq: gammarange} is  equal to the dimension of $\I^{\perp}(\delta_1,\delta_2)$. As $\delta_2$ tends to infinity, it is equivalent to the number of  linearly independent equations in \eqref{eq: gammarange} such that $\delta'_1 \leq \gamma \leq \delta_2$.
When $\gamma$, $\delta'_1 \leq \gamma \leq \delta_2$, is fixed, the number of linearly independent equations in \eqref{eq: gammarange} does not depend on $\gamma$.  Hence there exists a positive integer $c(\delta_1)$ which does not depend on $\delta_2$ such that
$$
\dim \I^{\perp}(\delta_1,\delta_2) \sim c(\delta_1)\delta_2  \,, \mbox{ as $\delta_2$ tends to infinity,}
$$
as wanted. A more detailed argument is provided in Section \ref{app:dimension}.
\end{proof}

\begin{lem}\label{lem:P2-majorationespacepolynome}
For every pair of non-negative integers $(\delta_1,\delta_2)$, one has 
$$
\dim \I^{\perp}(2\delta_1,\delta_2) \leq 2^{m^2}\dim \I^{\perp}(\delta_1,\delta_2) \,.
$$
\end{lem}

\begin{proof}
Every $P\in \Q[\Y,z]_{2\delta_1,\delta_2}$ can be decomposed as 
\begin{equation}\label{eq:P2-pl}
P(\Y,z)=\sum_{\ell=1}^{2^{m^2}} e_\ell(\Y)^{\delta_1}P_\ell(\Y,z) \,,
\end{equation}
where we let $e_1(\Y),\ldots,e_{2^{m^2}}(\Y)$ denote the $2^{m^2}$ distinct monomials of degree at most $1$ in each $y_{i,j}$, 
and where the polynomials $P_\ell(\Y,z)$ all belong to $\Q[\Y,z]_{\delta_1,\delta_2}$. If each polynomial $P_\ell$ belongs to $\I(\delta_1,\delta_2)$ then $P \in \I(2\delta_1,\delta_2)$. Hence, the decomposition \eqref{eq:P2-pl} defines a linear map 
$$
\left(\Q[\Y,z]_{\delta_1,\delta_2}/ \I(\delta_1,\delta_2)\right)^{2^{m^2}} \mapsto \Q[\Y,z]_{2\delta_1,\delta_2}/ \I(2\delta_1,\delta_2)
$$ that is surjective. The result follows. 
\end{proof}

\subsubsection{Nullstellensatz and relation matrices} 

In this section, we show how Hilbert's Nullstellensatz allows us to ensure the existence of a matrix $\bphi$, 
whose coordinates are all algebraic over $\Q(z)$, and  
which we call a \emph{relation matrix}. Such a matrix encodes the linear relations over $\Q(z)$ between  the functions  
$f_1(z),\ldots,f_m(z)$ and is the cornerstone of the proof 
of Theorem~\ref{thm:P2-liftinglinear}. 

\medskip

We first prove the following lemma. 

\begin{lem}\label{lem:P2-idealI}
The set $\mathcal I$ is a radical ideal of $\Q(z)[\Y]$. 
\end{lem}

\begin{proof} 
Checking that $\mathcal I$ is an ideal of $\Q(z)[\Y]$ is not difficult. If $P_1,P_2 \in \I$,
then $P_1+P_2$ vanishes at $(A_k(\alpha),\alpha^{q^k})$ for all $k\gg 1$ and hence $P_1+P_2 \in \I$. 
Now let $P_1 \in \I$ and $P_2\in \Q(z)[\Y]$. On the one hand, $P_1(A_k(\alpha),\alpha^{q^k})=0$ 
 for all $k \gg 1$ and 
$P_2(\Y,z)$ is well-defined at $(A_k(\alpha),\alpha^{q^k})$ for $k\gg 1$. We deduce that 
\[
P_1(A_k(\alpha),\alpha^{q^k})P_2(A_k(\alpha),\alpha^{q^k})=0\quad\quad  \forall k\gg 1\,.
\]
Hence $P_1P_2 \in \I$. Let $P\in \Q(z)[\Y]$ be such that $P^r \in \I$ for some $r$. If $k$ is a non-negative integer such that  $P(A_k(\alpha),\alpha^{q^k})^r =0$, then $P(A_k(\alpha),\alpha^{q^k})=0$.  
Hence $P \in \I$ and $\I$ is a radical ideal. 
\end{proof}

Throughout this article, we let  $\mathbb A\subset \bigcup_{d \geq 1}\Q((z^{1/d}))$ denote the algebraic closure of $\Q(z)$ in the field of Puiseux series. By the Newton-Puiseux Theorem, $\mathbb A$ is algebraically closed. 

\begin{lem}\label{lem:P2-phi1}
There exists a matrix $\bphi(z) \in  {\rm GL}_m(\mathbb A)$ such that 
$$
P(\bphi(z),z) = 0\, ,
$$ 
for all polynomials $P \in \I$.
\end{lem}

\begin{proof}
Let us consider the affine algebraic set $\mathcal V$ associated with the radical ideal $\I$. That is,   
$$
{\mathcal V}:=\{\bphi(z) \in \mathcal M_m(\mathbb A)\ : \ P(\bphi(z),z) = 0 \,, \; \forall P \in \I \}\, .
$$
According to the weak form of Hilbert's Nullstellensatz (see, for instance, \cite[Theorem 1.4, p.\ 379]{Lang}),  
$\mathcal V$ is non-empty as soon as  
$\I$ is a proper ideal of $\Q(z)[\Y]$. But the definition of $\mathcal I$ implies that non-zero constant polynomials do not belong to 
$\mathcal I$. Hence $\mathcal V$ is non-empty.  

Now, let us assume by contradiction that $\det \bphi(z)=0$ for all $\bphi(z)$ in $\mathcal V$. 
By Hilbert's Nullstellensatz (see, for instance, \cite[Theorem 1.5, p.\ 380]{Lang}), 
the polynomial $\det \Y$  belongs to the radical of the ideal  
$\I$. Hence  $\det \Y\in \I$  for $\I$ is radical. Thus,  $\det A_k(\alpha)=0$ 
for $k\gg 1$. This provides a contradiction since $A_k(\alpha)$ is invertible 
for all $k\geq 0$. We thus deduce that there exists an invertible matrix $\bphi(z)$ 
in $\mathcal V$, as wanted. 
\end{proof}

\begin{defi} A matrix $\bphi(z) \in  {\rm GL}_m(\mathbb A)$ satisfying the property of Lemma 
 \ref{lem:P2-phi1} is called a \emph{relation matrix}. 
\end{defi}

The next lemma plays a central role in the proof of Theorem  \ref{thm:P2-liftinglinear}.

\begin{lem}\label{lem:P2-phi2}
Let $\bphi(z) \in  {\rm GL}_m(\mathbb A)$ 
be a relation matrix.  
Then  
$$
P\left(\bphi(z) A_k(z),z^{q^k}\right) = 0\, ,
$$ 
for all $P \in \I$ and  all $k \geq 0$. 
\end{lem}

\begin{proof} 
 
Let $P \in \I$, $\bphi(z) \in  {\rm GL}_m(\mathbb A)$ be a relation 
matrix, and $k$ be a non-negative integer. Set $Q(\Y,z):=P(\Y A_k(z),z^{q^k}) \in \Q(z)[\Y]$. 
For every $\ell \gg 1$, the polynomial $Q(\Y,\alpha^{q^\ell})$ is well-defined and we have
$$
Q(A_\ell(\alpha),\alpha^{q^\ell}) 
= P(A_\ell(\alpha) A_k(\alpha^{q^\ell}),(\alpha^{q^\ell})^{q^k})
=P(A_{k+\ell}(\alpha),\alpha^{q^{k+\ell}}) = 0\, ,
$$
since $A_\ell(\alpha) A_k(\alpha^{q^\ell})=A_{k+\ell}(\alpha)$. Hence $Q \in \I$ and 
$$
P(\bphi(z) A_k(z),z^{q^k})=Q(\bphi(z),z)=0\,,
$$
as wanted. 
\end{proof}

\subsubsection{Analyticity and relation matrices}
We address now the question of the analyticity of relation matrices.

\begin{lem}
\label{lem:P2-zeroMasserextended} 
 Let $\bphi(z) \in  {\rm GL}_m(\mathbb A)$ be a relation matrix. Then the three  following properties holds for $k \gg 1$.  

\begin{itemize}

\item[{\rm (a)}] The point $\alpha^{q^{k}}$ belongs to the disc of convergence of each of the functions 
$f_{1}(z),\ldots,f_{m}(z)$. 

\item[{\rm (b)}] Each coordinate of $\bphi(z)$ defines an analytic function on some neighborhood of $\alpha^{q^{k}}$.

\item[{\rm (c)}] The matrix $\bphi(\alpha^{q^{k}})$ is invertible. 
\end{itemize}
\end{lem}

\begin{proof} 
Since 
$\lim_{k\to\infty} \alpha^{q^{k}} =0$ and  $f_1(z),\ldots,f_m(z)$ are analytic on some neighborhood of $0$, Property (a) holds for $k\gg 1$. Recall that an algebraic function has only finitely many singularities and finitely many zeros. Hence, for $k\gg 1$, $\alpha^{q^{k}}$ is neither a singularity of one of the coordinates of $\bphi(z)$ nor a zero of $\det \bphi(z)$. We deduce that Properties (b) and (c) hold for  $k\gg 1$. 
\end{proof}

\subsection{The key Lemma}\label{subsec:klemma}

Let 
$$
L(X_1,\ldots,X_m)=: \sum_{j=1}^m \tau_j X_i  \in \Q[X_1,\ldots,X_r]\,  
$$
be defined as in Theorem  \ref{thm:P2-liftinglinear}. 
Set $\btau:=(\tau_1,\ldots,\tau_m) \in \Q^m$ and $\X:=(X_1,\ldots,X_m)^\top$, so that $L(\X)=\btau\X$. 
Given a matrix of indeterminates 
$\Y:=(y_{i,j})_{1\leq i,j\leq m}$, we set 
\begin{equation*}
\label{eq:P2-defF}
F(\Y,z):=\sum_{i,j} \tau_i y_{i,j} f_j(z) = \btau \Y \f(z) \in \Q\{z\}[\Y]\, ,
\end{equation*}
where we recall that $\f(z) := (f_1(z),\ldots,f_m(z))^\top$. Note that $F$ is a linear form in $\Y$. 
Evaluating at $({\rm I}_m,\alpha)$, where ${\rm I}_m$ is the identity matrix of size $m$, we obtain that
\begin{equation}\label{eq:P2-annulationF}
F({\rm I}_m,\alpha)=\sum_{i=1}^m  \tau_if_i(\alpha) =L(\f(\alpha)) = 0 \,. 
\end{equation}

\begin{rem}\label{rem:P2-F} 
We have  $F(\Y,z)  \in \Q[\Y,\f(z)]\subset\Q\{z\}[\Y]$. Also, $F(\Y,z)$ can be seen as an element of $\Q[\Y][[z]]$, 
as we will sometimes do in what follows. 
\end{rem}

\subsubsection{Iterated relations}

For every $k \geq 0$, Equality \eqref{eq:P2-Mahleritere} implies the following equality in $\mathbb A[\Y]$:   
\begin{eqnarray}\label{eq:P2-iterationF}
\nonumber F(\Y,z) &=& \btau \Y  \f(z) \\
&=& \btau \Y A_k(z) \f(z^{q^k}) \\
\nonumber &=& F(\Y A_k(z),z^{q^k})\, .
\end{eqnarray}
The point $\alpha$ being regular with respect to \eqref{eq:P2-MahlerSystem},  
we deduce from \eqref{eq:P2-annulationF} 
 that  
\begin{equation}
\label{eq:P2-annulationFitere}
F(A_k(\alpha),\alpha^{q^k})=0 \,, \quad\quad \forall k \geq 0\,.
\end{equation}

\subsubsection{The matrices $\bTheta_k(z)$}\label{subsubsec:P2 theta_k} 

From now on, we fix a relation matrix ${\bphi}(z)$ and a non-negative integer $k_0$  
satisfying the properties of  Lemma 
 \ref{lem:P2-zeroMasserextended}. Set  
\begin{equation}\label{eq: k0}
\xi := \alpha^{q^{k_0}} \, .
\end{equation}
Item (a) in Lemma  \ref{lem:P2-zeroMasserextended} ensures the existence of a positive
real number $r_1<1$ such that $0<\vert \xi \vert<r_1$ and such that all the power series 
$f_{1}(z),\ldots,f_{m}(z)$ have a radius of convergence larger than $r_1$.  
Then, by  Item (b) in the same lemma, we can choose 
$r_2>0$ satisfying  $0< \vert \xi \vert +r_2< r_1$ and such that the 
coefficients of the matrix $\bphi(z)$ are analytic on the disc $\mathcal D(\xi,r_2)$. 
For every $k\geq k_0$, we set   
\begin{equation}\label{eq:P2-theta}
\bTheta_k(z):= A_{k_0}(\alpha)\bphi(\alpha^{q^{k_0}})^{-1}\bphi(z)A_{k-k_0}(z)\, 
\end{equation}
so that we have  $\bTheta_k(\xi) = A_k(\alpha)$, for every $k\geq k_0$.

\begin{rem}\label{rem:P2-thetak}
By Lemma  \ref{lem:P2-zeroMasserextended}, the coefficients of 
$\bTheta_{k_0}(z)$ are analytic on the disc   
$\mathcal D(\xi,r_2)$.  On the other hand, one has 
$$
\bTheta_k(z)=\bTheta_{k-1}(z) A(z^{q^{k-1-k_0}})\,, \quad\quad \forall k>k_0\,.
$$ 
This  implies that, for every $k\geq k_0$,  the coefficients of 
$\bTheta_{k}(z)$ are analytic on some neighborhood of $\xi$, that  is on some disc 
$\mathcal D(\xi,r_k)\subset \mathcal D(\xi,r_2)$.  
In what follows, we will consider the expression  
$F(\bTheta_k(z),z^{q^{k-k_0}})$. Formally, it is a polynomial in 
$f_1(z^{q^{k-k_0}}),\ldots,f_m(z^{q^{k-k_0}})$ and the coordinates of 
$\bTheta_k(z)$. 
Note that it also defines an analytic function on  
$\mathcal D(\xi,r_k)\subset \mathcal D(\xi,r_2)$. 
In addition, $F(\bTheta_{k_0}(z),z)$ is analytic on 
 $\mathcal D(\xi,r_2)$. 
Indeed, the functions $f_{1}(z),\ldots,f_{m}(z)$ are analytic on 
$\mathcal D(0,r_1)\supset\mathcal D(\xi,r_2)$, while our choice of  
$k_0$ ensures that the coordinates of   
$\bTheta_{k_0}(z)$ are analytic on $\mathcal D(\xi,r_2)$. 
\end{rem}

\subsubsection{The key lemma}\label{sec:P2-cle}
 
The end of the section is devoted to proof of the following result.

\begin{lem}
\label{lem:P2-nulliteF}
One has  
$F(\bTheta_{k_0}(z),z)=0$.
\end{lem}

 In what follows, we argue by contradiction, assuming that 
\begin{equation}\label{eq:P2-notzero}
F(\bTheta_{k_0}(z),z)\not=0\,.
\end{equation}
We divide the proof of Lemma  \ref{lem:P2-nulliteF} into the four steps (AF), (UB), (NV), and (LB), 
following the classical proof scheme in transcendental number theory. In Step (AF) we build an \textit{auxiliary function} by considering some sort of Padé approximant of type I for the first powers of $F(\Y,z)$. In Step (UB) we compute some \textit{upper bound} for the absolute value of the evaluation of our auxiliary function at $(A_k(\alpha),\alpha^{q^k})$, for large $k$, by means of analytic estimates. In Step (NV) we prove that our auxiliary function is \textit{non-vanishing} at $(A_k(\alpha),\alpha^{q^k})$ for infinitely many $k$. In Step (LB), 
we provide a \textit{lower bound} for the absolute value of the evaluation of our auxiliary function at $(A_k(\alpha),\alpha^{q^k})$, for infinitely many $k$, by  using Liouville's inequality. Finally, we show that the steps (UB) and (LB) lead to a contradiction.

\medskip
\noindent \textbf{Step (AF).}
Given a formal power series $E:=\sum_{\lambda\geq 0} e_\lambda(\Y)z^\lambda \in \Q[\Y][[z]]$ 
and an integer $p>0$, we let  
$$
E_p:=\sum_{\lambda=0}^{p-1} e_\lambda(\Y)z^\lambda \in \Q[\Y,z]
$$
denote the truncation of  $E$ at order $p$ with respect to $z$. 
We recall that  
$\I^\perp(\delta_1,\delta_2)$ is a complement to  
$\I(\delta_1,\delta_2)$ in $\Q[\Y,z]_{\delta_1,\delta_2}$. 

\begin{lem}
\label{lem:P2-fonctionauxiliaire}
Let $\delta_1\geq 0$ and $\delta_2\gg \delta_1$ be two integers. 
Let  $p:=\left\lfloor \frac{\delta_1\delta_2}{2^{m^2+2}}\right\rfloor$. 
Then there exist polynomials $P_i\in \I^{\perp}(\delta_1,\delta_2)$, 
$0 \leq i \leq \delta_1$, not all zero,  
 such that the formal power series  
$$
E(\Y,z) := \displaystyle\sum_{j=0}^{\delta_1} P_j(\Y,z)F(\Y,z)^j \in \Q[\Y][[z]]
$$ 
satisfies 
$E_p(\bTheta_k(z),z^{q^{k-k_0}})=0$ for all $k \geq  k_0$.
\end{lem}

\begin{proof}
Set  
$$
\J(\delta_1,\delta_2) := 
\{ P \in \Q[z,\Y]\ : \ P( A_{k_0}(\alpha)\bphi(\alpha^{q^{k_0}})^{-1}\Y,z )\in \I(\delta_1,\delta_2)\}\,  .
$$
The $\Q$-vector spaces  $\J(\delta_1,\delta_2)$ and $\I(\delta_1,\delta_2)$ have same dimension.  
This  follows directly from the fact that the map  
$$
\begin{array}{ccc} \Q[\Y,z]_{\delta_1,\delta_2} & \rightarrow & \Q[\Y,z]_{\delta_1,\delta_2}
\\ P(\Y,z) & \mapsto & P( A_{k_0}(\alpha)\bphi(\alpha^{q^{k_0}})^{-1}\Y,z )\, 
\end{array} 
$$
is an isomorphism, the matrix $ A_{k_0}(\alpha)\bphi(\alpha^{q^{k_0}})^{-1}$ being invertible.  
Furthermore, we have  
\begin{equation}\label{eq:P2-thetazero}
P(\bTheta_k(z),z^{q^{k-k_0}})=0\,, \quad\quad  \forall P\in \J(\delta_1,\delta_2),\; \forall k\geq k_0\,.
\end{equation} 
Indeed, if 
 $P \in \J(\delta_1,\delta_2)$, then  
$P( A_{k_0}(\alpha)\bphi(\alpha^{q^{k_0}})^{-1}\Y,z )\in \I(\delta_1,\delta_2)$, and Lemma 
 \ref{lem:P2-phi2} implies that  
$$
P( A_{k_0}(\alpha)\bphi(\alpha^{q^{k_0}})^{-1}\bphi(z)A_k(z),z^{q^k} )=0\, , 
\quad\quad \forall k \geq 0\,.
$$
For $k \geq  k_0$, replacing $k$ by $k-k_0$ in the previous equality, we obtain that 
$$
P( A_{k_0}(\alpha)\bphi(\alpha^{q^{k_0}})^{-1}\bphi(z) A_{k-k_0}(z),z^{q^{k-k_0}} )=0\, .
$$
By \eqref{eq:P2-theta}, we thus have $P(\bTheta_k(z),z^{q^{k-k_0}})=0$.

Let $p$ be as in the lemma and let us consider  
the three $\Q$-linear maps: 
\begin{eqnarray*}
&\left\{\begin{array}{cc}
\left(\I^{\perp}(\delta_1,\delta_2)\right)^{\delta_1+1}
\\
(P_0(\Y,z),\ldots,P_{\delta_1}(\Y,z))
\end{array} \right.&
\\
& \big{\downarrow}&
 \\
& \left\{
\begin{array}{c} \Q[\Y]_{2\delta_1}[[z]]
\\ E(\Y,z) := \sum_{j=0}^{\delta_1} P_j(\Y,z)F(\Y,z)^j
\end{array} \right.&
\\
&\big{\downarrow}&
\\
&\left\{\begin{array}{c} \Q[\Y,z]_{2\delta_1,p-1} 
\\ 
 E_p(\Y ,z)
 \end{array}\right.&
 \\
&\big{\downarrow}&
\\
&\left\{\begin{array}{c}  \Q[\Y,z]_{2\delta_1,p-1}/\J(2\delta_1,p-1) 
\\ 
 E_p(\Y,z) \mod \J(2\delta_1,p-1)
 \end{array}\right.&
\end{eqnarray*}
Note that these maps are well-defined.  
By Lemma  \ref{lem:P2-dimensionespace}, the dimension of the $\Q$-vector space $\I^\perp(\delta_1,\delta_2)$
is at least equal to $\frac{c_1(\delta_1)}{2}\delta_2$, assuming  that $\delta_2$ is large enough.  
It follows that 
\begin{equation}\label{eq:P2-dimev}
\dim_{\Q}\left(\left(\I^{\perp}(\delta_1,\delta_2)\right)^{\delta_1+1}\right)\geq 
 \frac{c_1(\delta_1)}{2}(\delta_1+1)\delta_2 \,.
\end{equation}
For every pair of non-negative integers $(u,v)$, set 
$$
\overline{\J}(u,v):= \Q[\Y,z]_{u,v}/\J(u,v)\,.
$$  
Since $\J(\delta_1,\delta_2)$ and $\I(\delta_1,\delta_2)$ have same dimension, 
Lemma   \ref{lem:P2-majorationespacepolynome} implies that 
$$
\dim_{\Q} \overline{\J}(2\delta_1,p-1)\leq 
2^{m^2} \dim_{\Q} \overline{\J}(\delta_1,p-1)\, .
$$
Now, if  $\delta_2$ is sufficiently large, Lemma   \ref{lem:P2-dimensionespace} ensures that 
$$
\dim_{\Q} \overline{\J}(\delta_1,p-1) \leq 2c_1(\delta_1)p \,.
$$ 
On the other hand, the choice of $p$ ensures that 
$$
2^{m^2}\left(2c_1(\delta_1)p\right) <  \frac{c_1(\delta_1)}{2}(\delta_1+1)\delta_2
$$ 
and \eqref{eq:P2-dimev} implies that   
$$
\dim_{\Q}\left( \left(\I^{\perp}(\delta_1,\delta_2)\right)^{\delta_1+1}\right) > \dim_{\Q} \left( \Q[\Y,z]_{2\delta_1,p-1}/\J(2\delta_1,p-1)\right)\,.
$$
Hence the $\Q$-linear map defined by  
$$(P_0(\Y,z),\ldots,P_{\delta_1}(\Y,z))\mapsto E_p(\Y,z)\mod \J(2\delta_1,p-1)$$ 
has a non-trivial kernel.  
We deduce the existence of polynomials $P_0,\ldots,P_{\delta_1}$ in $\I^{\perp}(\delta_1,\delta_2)$, 
not all zero, 
such that 
$E_p \in \J(2\delta_1,p-1)$.  
By \eqref{eq:P2-thetazero}, we obtain that $E_p(\bTheta_k(z),z^{q^{k-k_0}})=0$ for all $k \geq  k_0$.  
This ends the proof. 
\end{proof}

Let $E\in \Q[\Y][[z]]$ be a formal power series satisfying the properties of  Lemma  \ref{lem:P2-fonctionauxiliaire} and 
let $v_0$ be the smallest 
index such that  the polynomial $P_{v_0}$ is non-zero.  Then the formal power series  
\begin{equation*}\label{eq:P2-E}
\mathfrak E(\Y,z) := \sum_{j\geq v_0} P_j(\Y,z)F(\Y,z)^{j-v_0} \in \Q[\Y][[z]]\,
\end{equation*}
is the auxiliary function that we were looking for.  Note that  we have 
\begin{equation}\label{eq:P2-EE'}
\mathfrak E(\Y,z)F(\Y,z)^{v_0} = E(\Y,z)\,.
\end{equation}

\medskip \noindent
\textbf{Warning.}
The function $\mathfrak E(\bTheta_k(z),z^{q^{k-k_0}}) $ can be though of as a 
simultaneous Pad\'e approximant of  type I for the first 
$\delta_1$th powers of  $F(\bTheta_k(z),z^{q^{k-k_0}})$. However, we have to be careful:  
$F(\bTheta_k(z),z^{q^{k-k_0}})$ it is not necessarily a power series in $z$. It is a linear combination of  
$f_1(z^{q^{k-k_0}}),\ldots,f_m(z^{q^{k-k_0}})$ 
whose coefficients are only known to be algebraic over $\Q(z)$. We only know that 
$F(\bTheta_k(z),z^{q^{k-k_0}})$ is analytic in some neighborhood of the point $\xi$.

 \medskip
 \noindent \textbf{Step (UB).}
 The aim of this step is to prove that there exists a real number $c_2>0$ such that  
 \begin{equation}\label{eq:P2-majprinc}
 \vert \mathfrak E(A_k(\alpha),\alpha^{q^k}) \vert \leq e^{-c_2q^k\delta_1\delta_2 }\, ,
 \quad\quad  \forall k \gg \delta_2 \gg \delta_1\,.
 \end{equation}
 
 \medskip
 
 According to Remark \ref{rem:P2-thetak}, the  functions  
 $\mathfrak E(\bTheta_k(z),z^{q^{k-k_0}})$, $F(\bTheta_{k_0}(z),z)^{v_0}$, and 
 $E(\bTheta_k(z),z^{q^{k-k_0}})$ are 
 all analytic  on the disc $\mathcal D(\xi, r_k)$. Hence they respectively have power series expansions of the form  
 \begin{eqnarray}
 \label{eq:P2-developpementE}
 \quad \mathfrak E(\bTheta_k(z),z^{q^{k-k_0}}) &=:&\sum_{\lambda =0}^{+\infty} 
 e_{\lambda,k} (z-\xi)^{\lambda}\, , \quad e_{\lambda,k} \in \C\,, \\
 \label{eq:P2-developpementF}
 F(\bTheta_{k_0}(z),z)^{v_0}&=:& \sum_{\lambda =0}^{+\infty} a_{\lambda}(z-\xi)^\lambda\, , \quad a_{\lambda} \in \C\,,\\
 \label{eq:P2-developpement_E'}
 \quad E(\bTheta_k(z),z^{q^{k-k_0}})&=:& 
 \sum_{\lambda =0}^{+\infty} \epsilon_{\lambda,k} (z-\xi)^{\lambda}\, , \quad \epsilon_{\lambda,k} \in \C\,.
 \end{eqnarray}
 
 We need the following result whose proof is postponed after the end of the argument for proving our main upper bound \eqref{eq:P2-majprinc}.

 \begin{lem}\label{lem:P2-majo_epsilonk} Let $p$ be defined as in Lemma \ref{lem:P2-fonctionauxiliaire}. 
 	There exists a real number $\gamma>0$ that does not depend on the integers  
 	$\delta_1,\delta_2$, $\lambda$, and $k$,  and such that 
 	$$
 	\vert \epsilon_{\lambda,k}\vert \leq e^{-\gamma q^k p}\, ,\quad\quad \forall k \gg \delta_2 \gg \delta_1,\lambda\,.
 	$$
 \end{lem}

 Using \eqref{eq:P2-iterationF}, we get that  
 $$
 F(\bTheta_k(z),z^{q^{k-k_0}}) = F(\bTheta_{k_0}(z),z)\,, \;\; \forall k\geq k_0\,.
 $$ 
 By \eqref{eq:P2-EE'},  we thus have 
 \begin{equation}\label{eq:P2-factorisationE'}
 \mathfrak E(\bTheta_k(z),z^{q^{k-k_0}})F(\bTheta_{k_0}(z),z)^{v_0}=E(\bTheta_k(z),z^{q^{k-k_0}}) \,,
 \end{equation}
 for all $k \geq  k_0$ and all $z\in\mathcal D(\xi,r_k)$. 
 We use now our assumption that $F(\bTheta_{k_0}(z),z)$ is non-zero (see \eqref{eq:P2-notzero}). There thus exists at least  
 one non-zero coefficient $a_{\lambda}$ in \eqref{eq:P2-developpementF}. Let us consider the least integer $\lambda_0$ 
 such that $a_{\lambda_0}\not= 0$. 
 Identifying the coefficients of   
 $(z - \xi)^{\lambda_0}$ in the power series expansion of both sides of  
 \eqref{eq:P2-factorisationE'} with the help of  \eqref{eq:P2-developpementE},  \eqref{eq:P2-developpementF}, and  \eqref{eq:P2-developpement_E'}, 
 we obtain that  
 \begin{equation}\label{eq: eeps}
 e_{0,k}a_{\lambda_0} = \epsilon_{\lambda_0,k}\, , \;\; \forall k\geq k_0\,.
 \end{equation}
 Since  $\bTheta_{k}(\xi)=A_k(\alpha)$ (see \eqref{eq:P2-theta}) and $a_{\lambda_0}$ depends only on $\delta_1$ but not on $k$, we infer from 
 Lemma~\ref{lem:P2-majo_epsilonk}, Equality \eqref{eq: eeps},  
 and the definition of $p$ (see Lemma~\ref{lem:P2-fonctionauxiliaire}), 
 the existence of a  real number  $c_2>0$ that does not depend on  
 $\delta_1$, $\delta_2$, and $k$, 
 such that 
 \begin{eqnarray}\label{eq:P2-majoEb}
 \nonumber  \left\vert \mathfrak E(A_k(\alpha),\alpha^{q^k})\right\vert &=&
 \left\vert \mathfrak E(\bTheta_k(\xi),\xi^{q^{k-k_0}}) \right\vert \\
 \nonumber  &=& \vert e_{0,k} \vert 
 \\ \nonumber &=&\vert \epsilon_{\lambda_0,k}\vert/\vert a_{\lambda_0} \vert 
 \\ \nonumber & \leq & e^{-c_2 q^k\delta_1\delta_2}\, ,\quad\quad \forall k\gg \delta_2 \gg \delta_1\,.
 \end{eqnarray}
 This proves the upper bound \eqref{eq:P2-majprinc}, as wanted.

 \medskip 
 
 Now, it remains to prove Lemma \ref{lem:P2-majo_epsilonk}. 
 
 \begin{proof}[Proof of Lemma \ref{lem:P2-majo_epsilonk}] 
 	Set 
 	$$
 	G(\Y,z):=E(\Y,z) - E_p(\Y,z)\in \Q\{z\}[\Y] \, ,
 	$$ 
 	where $p$ is defined as in Lemma  \ref{lem:P2-fonctionauxiliaire}.
 	By Lemma  \ref{lem:P2-fonctionauxiliaire}, we have  
 	\begin{equation}\label{eq:P2-GE'}
 	G(\bTheta_k(z),z^{q^{k-k_0}})=E(\bTheta_k(z),z^{q^{k-k_0}}) \,,\quad\quad \forall k\geq k_0\,.
 	\end{equation} 
 	Let $\bnu_1,\ldots,\bnu_s$ denote an enumeration of all the $m \times m$ matrices with coefficients in the set  $\{0,1,\ldots,2\delta_1\}$. 
 	There exists a unique decomposition of the form 
 	$$
 	G(\Y,z) =: \sum_{i=1}^s \sum_{ \lambda = p}^\infty g_{\lambda,i} z^\lambda \Y^{\bnu_i}\, ,
 	$$
 	where $g_{\lambda,i} \in \Q$. 
 	For every $i$, $1\leq i \leq s$, we define the formal power series 
 	$$
 	G_i(z):=\sum_{\lambda = p}^\infty g_{\lambda,i} z^\lambda \in \Q[[z]]\, .
 	$$
 	By definition of $F(\Y,z)$, these series belong to $\Q[z,\f(z)]$. In particular, they are analytic on some 
 	disc  $\mathcal D( 0,r)$ with $r>r_1$ (where $r_1$ is defined at the beginning of Section \ref{subsubsec:P2 theta_k}). 
 	From the Cauchy-Hadamard inequality, there exists a positive real number $\gamma_1(\delta_1,\delta_2)$ 
 	such that 
 	\begin{equation}\label{eq:P2-majoration_g}
 	\vert g_{\lambda,i} \vert \leq \gamma_1(\delta_1,\delta_2)r_1^{-\lambda }\, , \quad\quad \forall \lambda \geq 0\,.
 	\end{equation}
 	For every $k \geq k_0$, $G_i(z^{q^{k-k_0}})$ can thus be written as 
 	\begin{equation}\label{eq:P2-developpement_G_Tk_0}
 	G_i(z^{q^{k-k_0}})=: \sum_{ \lambda =  q^{k-k_0} p}^\infty g_{\lambda,i,k} z^\lambda\, ,
 	\end{equation}
 	with  $g_{\lambda,i,k} \in \Q$. Furthermore, this power series is absolutely convergent 
 	on the disc $\mathcal D( 0,r_1)$. Since $r_1 \leq 1$, we deduce from  \eqref{eq:P2-majoration_g} that  
 	\begin{equation}
 	\label{eq:P2-majoration_gk}
 	\vert g_{\lambda,i,k} \vert \leq \gamma_1(\delta_1,\delta_2)r_1^{- \lambda  q^{k_0-k}  } \leq \gamma_1(\delta_1,\delta_2)r_1^{- \lambda  }  \, ,
 	\end{equation}
 	for all $\lambda\geq 0$, $i\in\{1,\ldots,s\}$, and $k \geq k_0$. 
 	On the other hand, every function  
 	$G_i(z^{q^{k-k_0}})$, $1\leq i \leq s$, $k \geq k_0$, is analytic on the disc $\mathcal D(\xi,r_k)$. 
 	Thus, we can write 
 	\begin{equation}\label{eq:P2-developpement_G_alpha}
 	G_i(z^{q^{k-k_0}})=: \sum_{\lambda =0}^\infty h_{\lambda,i,k} (z - \xi)^\lambda \, ,
 	\end{equation}
 	where $h_{\lambda,i,k} \in \C$. 
 	Since by assumption $\mathcal D(\xi,r_k) \subset \mathcal D( 0,r_1)$, 
 	the two power series expansions \eqref{eq:P2-developpement_G_Tk_0} and \eqref{eq:P2-developpement_G_alpha} 
 	match on $\mathcal D(\xi,r_k)$. Using the equality  	
	\begin{equation}\label{eq:dev_z_xi}
 	z^{\gamma}=((z-\xi)+ \xi)^\gamma = \sum_{\lambda =0}^{\gamma} \binom{\gamma}{\lambda} \xi^{\gamma-\lambda} (z-\xi)^{\lambda}
 	\end{equation}
 	and identifying, for every $\lambda \geq 0$, the coefficients of   
 	$(z - \xi)^{\lambda}$ in \eqref{eq:P2-developpement_G_Tk_0} and \eqref{eq:P2-developpement_G_alpha}, 
 	we deduce that 
 	\begin{equation}\label{eq:P2-identit_gk}
 	h_{\lambda,i,k} = \sum_{\substack{ \gamma  \geq  q^{k-k_0} p 
 			\\ \gamma \geq \lambda}} \binom{\gamma}{\lambda} g_{\gamma,i,k} \xi^{\gamma-\lambda}  \, .
 	\end{equation} 
 	For $\gamma \geq \lambda$, one has
 	\begin{equation}\label{eq:P2-majbinom}
 	\binom{\gamma}{\lambda} =  \frac{\gamma!}{(\gamma-\lambda)!\lambda!}
 	\leq \gamma^{\lambda} \, .
 	\end{equation}
 	Given $\lambda \geq 0$, we have that 
 	$ \lambda  < q^{k-k_0}p$ as soon as $k$ is large enough, and  
 	since $ \vert \xi \vert<r_1$, we infer from
 	\eqref{eq:P2-majoration_gk} and \eqref{eq:P2-identit_gk}
 	the existence of a real number $\gamma_2>0$ that  does not
 	depend on $\delta_1$, $\delta_2$, $\lambda$, and $k$,   
 	such that
 	\begin{equation}\label{eq:P2-majoration_g_alpha}
 	\vert h_{\lambda,i,k}\vert \leq   e^{-\gamma_2q^k p}\, , \quad\quad  \forall k \gg \delta_1,\delta_2,\lambda\,.
 	\end{equation}

 	Now, we proceed to bound the absolute value of the coefficients of the power series expansion in $z-\xi$ 
	of  $\bTheta_k(z)^{\bnu_i}$, $1\leq i \leq s$.
%
 	Given a power series $Q(z)\in \Q\{z\}$ and $k \geq 0$, 
 	we write 
 	$$
 	Q(z^{q^{k-k_0}})=: \sum_{\lambda = 0 }^\infty q_{\lambda,k}(z-
 	\xi)^\lambda\, .
 	$$
 	For all $k$ large enough, $\xi^{q^{k-k_0}}$ belongs to the domain of analyticity of $Q(z)$. Using again \eqref{eq:dev_z_xi} and \eqref{eq:P2-majbinom} we obtain that, for every $\lambda\geq 0$,
 	$\vert q_{\lambda,k} \vert = \mathcal O(1)$ as $k$ tends to infinity, where 
 	the underlying constant in the $\mathcal O$ notation depends both on $Q(z)$ and 
 	$\lambda$. Fix some $\lambda \geq 0$. Let $v\geq 0$ be an integer such that the coordinates of $z^vA(z)$ have no poles at $0$. The coordinates 
 	of $z^{v}A(z)$ are convergent power series at $0$, and the points $\xi^{q^{k-k_0}}$ belong to their domain of analyticity for $k$ large enough. 
	Then, the coefficients of $(z-\xi)^\lambda$ in the power series expansion in $z-\xi$ of each of 
 	the coordinates 
 	of $z^{vq^{k-k_0}}A(z^{q^{k-k_0}})$  belong to $\mathcal O(1)$ as $k$ tends to infinity. Using \eqref{eq:dev_z_xi}, we write
 	\begin{align*}
 	z^{-vq^{k-k_0}} =&\, \xi^{-vq^{k-k_0}} \left(1+ \sum_{\lambda= 1}^{vq^{k-k_0}} \binom{vq^{k-k_0}}{\lambda} \xi^{-\lambda} (z-\xi)^{\lambda}\right)^{-1}
 	\\ = &\, \xi^{-vq^{k-k_0}} + \\   &\quad \quad \sum_{\lambda=1}^\infty \left(  \xi^{-vq^{k-k_0}}\sum_{t=1}^\lambda \sum_{ \lambda_1+\cdots+\lambda_t= \lambda}  \prod_{i=1}^t \binom{vq^{k-k_0}}{\lambda_i} \xi^{-\lambda_i} \right) (z-\xi)^\lambda 
 		\\ =:&\, \sum_{\lambda=0}^\infty r_{\lambda,k} (z-\xi)^\lambda\,.
 	\end{align*}
 Using \eqref{eq:P2-majbinom}, we deduce the existence of a real number $\gamma_3>0$ which does not depend on $k$ and such that 
 $\vert r_{\lambda,k} \vert = \mathcal O(e^{\gamma_3q^k})$ as $k$ tends to infinity. It follows that the absolute value of the coefficient of $(z-\xi)^\lambda$,  in the power series expansion in $z-\xi$ of each of 
 	the coordinates 
 	of $A(z^{q^{k-k_0}})$, 
 	 belongs to $\mathcal O(e^{\gamma_3q^k})$ as $k$ tends to infinity,  where the underlying constant in the $\mathcal O$ notation depends on $\lambda$ but not on $\delta_1$, $\delta_2$, and $k$.

 	By Remark  \ref{rem:P2-thetak}, the monomial $\bTheta_k(z)^{\bnu_i}$
 	is analytic on  $\mathcal D(\xi,r_k)$ for every $i$, $1\leq i \leq s$,
 	and every $k \geq k_0$. 
 	Thus,  we can write 
 	\begin{eqnarray}\label{eq:P2-developpement_thetak}
 	\bTheta_k(z)^{\bnu_i}   
 	& =:&
 	\sum_{\lambda=0}^{+\infty} \theta_{\lambda,i,k}(z-\xi)^{\lambda}\, ,
 	\end{eqnarray}
 	where $\theta_{\lambda,i,k} \in \C$.     
 	Using the recurrence relation
 	$$
 	\bTheta_{k+1}(z)=\bTheta_k(z)A(z^{q^{k-k_0}}) \,,
 	$$
 	we obtain the existence of a real number 
 	$\gamma_4(\lambda)>0$ that does not depend on $i$, $\delta_1$, $\delta_2$, and $k$, such that
 	the absolute value of the coefficient of $(z-\xi)^\lambda$ in each of 
 	the coordinates 
 	of $\bTheta_k(z)$ is at most $e^{\gamma_4(\lambda)q^k}$. Since $\vert \bnu_i \vert \leq 2m^2\delta_1$ for each $i$, there exists a real number $\gamma_5(\lambda)>0$ that does not depend on $i$, $\delta_1$, $\delta_2$, and $k$, such that
 	\begin{equation}\label{eq:P2-theta_ik}
 	\vert  \theta_{\lambda,i,k} \vert < e^{\gamma_5(\lambda)\delta_1 q^k} \,, \quad\quad \forall i,\, 1\leq i \leq s, \, \forall k \geq k_0\,.
 	\end{equation}
 	From  \eqref{eq:P2-developpement_E'}, \eqref{eq:P2-GE'}, \eqref{eq:P2-developpement_G_alpha}, and \eqref{eq:P2-developpement_thetak}, we deduce  that 
 	\begin{equation}\label{eq:P2-hthetae}
 	\sum_{i=1}^s \left(\sum_{\lambda =0}^{+\infty} h_{\lambda,i,k} (z-\xi)^\lambda \right)\left(\sum_{\lambda =0}^{+\infty} \theta_{\lambda,i,k} (z-\xi)^\lambda \right)  =\sum_{\lambda = 0}^{+\infty} \epsilon_{\lambda,k} (z-\xi)^{\lambda} \,.
 	\end{equation}
 	Finally, identifying the coefficents of $(z-\xi)^{\lambda}$ in both sides of 
 	\eqref{eq:P2-hthetae}, we have  
 	$$
 	\epsilon_{\lambda,k} = \sum_{i=1}^s \sum_{\gamma=0}^{\lambda} h_{\gamma,i,k}\theta_{\lambda-\gamma,i,k}\, .
 	$$
 	Note that  $p\gg \delta_1$ when $\delta_2 \gg \delta_1$. 
 	Inequalities  \eqref{eq:P2-majoration_g_alpha} and \eqref{eq:P2-theta_ik} imply the  existence of a real number $\gamma_6>0$ that does not depend on $\delta_1$, $\delta_2$, $\lambda$, and $k$, and such that 
 	$$
 	\vert \epsilon_{\lambda,k}\vert \leq 
 	e^{-\gamma_6 q^k p}\, ,\quad \quad\forall k\gg \delta_2\gg\delta_1,\lambda\,.
 	$$
 	Setting $\gamma:=\gamma_6$, this ends the proof. 
 \end{proof}

\medskip
\noindent \textbf{Step (NV).}
Let us first recall that by \eqref{eq:P2-annulationFitere} we have 
$$
F(A_k(\alpha),\alpha^{q^k})=0 \, , \quad\quad \forall k \geq 0\,.
$$
By construction of our auxiliary function, we deduce that 
$$\mathfrak E(A_k(\alpha),\alpha^{q^k})=P_{v_0}(A_k(\alpha),\alpha^{q^k})\,.$$ 
Furthermore, since this construction ensures that $P_{v_0} \notin \mathcal I$, 
there exists an infinite set of positive integers 
$\mathcal E$ such that 
$$
P_{v_0}(A_k(\alpha),\alpha^{q^k}) \not=0 \, , \quad\quad \forall k\in\mathcal E\,.
$$
Without any loss of generality, we assume that $k\geq k_0$ for all  $k\in\mathcal E$. 

\medskip
\noindent \textbf{Step (LB).} Given an algebraic number $\beta$, we let $h(\beta)$ denote the absolute logarithmic Weil height of $\beta$ 
(see \cite[Chapter 3]{Miw} or Section  \ref{app:height} for a definition). In order to prove our lower bound, we only need the following basic properties of the Weil height. 
The use of the Weil height simplifies some computations but any other standard notion of height would also do the job. 
Given two algebraic numbers $\beta$ and $\gamma$, one has (see \cite[Property 3.3]{Miw}):
\begin{eqnarray}\label{eq: height}
\nonumber h(\beta+\gamma) &\leq& h(\beta)+h(\gamma)+\log 2\\
h(\beta\gamma)&\leq& h(\beta) +h(\gamma)\\
\nonumber h(\beta^n)&=& \vert n\vert  h(\beta), \quad \beta\not=0, \;n\in \mathbb Z\,. 
\end{eqnarray}
Let $P:=\sum_{\k\in K} a_{\k}\X^\k\in\Q[X_1,\ldots,X_n]$, and $\beta_1,\ldots,\beta_n \in \Q$, we deduce from \cite[Lemma 3.7]{Miw} that 
\begin{equation}\label{eq: height3}
h(P(\beta_1,\ldots,\beta_n))\leq \sum_{i=1}^n \log(1+\deg_{X_i}(P)) + \sum_{i=1}^n(\deg_{X_i}P)h(\beta_i) + \sum_{\k \in K}h(a_\k)\, .
\end{equation}
Given  a number field ${\bf k}$, we have the fundamental \emph{Liouville inequality} (see \cite[p.\ 82]{Miw}): 
\begin{equation}\label{eq: liouville}
\log \vert \beta\vert \geq -[{\bf k}:\mathbb Q]h(\beta)\,, \quad\quad \forall \beta\not=0\in {\bf k} \, .
\end{equation}

We are going to use \eqref{eq: liouville} to find a lower bound for $\vert \mathfrak E(A_k(\alpha),\alpha^{q^k})\vert$. A 
simple computation by induction on $k$ shows that the height of each coordinate of $A_k(\alpha)$ is at most $\gamma q^k$ for some $\gamma>0$ that does not depend on $k$ (see Section  \ref{app:height} for more detail). 
The polynomial $P_{v_0}(\Y,z)$ has degree at most $\delta_1$ in each indeterminate $y_{i,j}$ and degree at most  
$\delta_2$ in $z$. Furthermore, its coefficients are algebraic numbers which only depend  on the parameters $\delta_1$ and $\delta_2$. 
Using  \eqref{eq: height} and \eqref{eq: height3}, we obtain that 
 the height of the algebraic number 
$P_{v_0}(A_k(\alpha),\alpha^{q^k})$ 
is at most $cq^k\delta_2$ for some constant $c$ that does not depend on $\delta_1$, $\delta_2$, and $k$, assuming 
that $k\gg \delta_2\geq \delta_1$. 
Since these algebraic numbers belong to a fixed number field,  Liouville's inequality 
ensures the existence of  $c_{3}>0$ that does not depend on $\delta_1$, $\delta_2$, and $k$, and such that 
\begin{equation}
\label{eq:P2-minorationE}
\vert \mathfrak E(A_k(\alpha),\alpha^{q^k})\vert=|P_{v_0}(A_k(\alpha),\alpha^{q^k})| \geq  
e^{-c_{3}q^k\delta_2}\, , \quad\quad\forall k\in\mathcal E, k\gg \delta_2 \geq \delta_1\,. 
\end{equation}

We are now ready to end the proof of our key lemma.

\begin{proof}[Proof of Lemma \ref{lem:P2-nulliteF}.]
By Inequalities \eqref{eq:P2-majprinc} and \eqref{eq:P2-minorationE}, we obtain that   
$$
e^{-c_{3}q^k \delta_2} \leq \vert \mathfrak E(A_k(\alpha),\alpha^{q^k})\vert  
\leq e^{-c_{2}q^k \delta_1\delta_2 }\, ,\quad\quad \forall k\in \mathcal E, \, k\gg \delta_2\gg \delta_1\,.
$$
We deduce that 
$$
c_{3} \geq c_{2}\delta_1\, .
$$
Since $c_{2}$ and $c_{3}$ are positive numbers which do not depend on $\delta_1$, 
this provides a contradiction, as soon as $\delta_1$ is large enough.   
\end{proof}

\subsection{End of the proof of Theorem \ref{thm:P2-liftinglinear}}\label{sec:P2-endproof}
The coordinates of $\bphi(z)$ being algebraic over $\Q(z)$,  they 
generate a finite extension of $\Q(z)$. Let $\mathbf k\subset \mathbb A$ denote this extension 
and let $\gamma\geq 1$ be the degree of $\mathbf k$.   We recall that $\mathbb A$ is the algebraic closure of $\Q(z)$ in the field of Puiseux series. 
Choosing a primitive element $\varphi(z)$ in $\mathbf k$,  
we obtain a decomposition of the form  
\begin{equation}\label{eq:P2-decompositionPhi}
\bphi(z)=: \sum_{j=0}^{\gamma-1}  \bphi_j(z) \varphi(z)^{j}\, ,
\end{equation}
where the matrices $\bphi_j(z)$, $0 \leq j \leq \gamma-1$, have coefficients in $\Q(z)$. Let $d(z)\in\Q[z]$ denote a common denominator 
of the coordinates of the matrices $\bphi_j(z)$. Without any loss of generality, we can assume that in \eqref{eq: k0} the integer $k_0$ has been chosen large enough so that 
$\varphi(z)$ is analytic at $\xi=\alpha^{q^{k_0}}$ and $d(\alpha^{q^{k_0}})\neq 0$. Let $q(z)$ denote the least common multiple 
of the denominators of the coordinates of the matrix  
$ A_{k_0}^{-1}(z)$. Since $\alpha$ is assumed to be regular with respect to the Mahler system~\eqref{eq:P2-MahlerSystem}, 
we have that $q(\alpha)\not=0$. 

By Lemma  \ref{lem:P2-nulliteF}, we know that $F(\bTheta_{k_0}(z),z)=0$, and substituting 
$z^{q^{k_0}}$ for  $z$, we obtain that 
$F(\bTheta_{k_0}(z^{q^{k_0}}),z^{q^{k_0}})=0$. 
The function $F(\Y,z)$ being linear in $\Y$, we deduce that 
$$
F \left( \frac{d(z^{q^{k_0}})q(z)}
{d(\alpha^{q^{k_0}})q(\alpha)}\bTheta_{k_0}(z^{q^{k_0}}),z^{q^{k_0}}\right)=0 \, .
$$ 
Writing $\bTheta_{k_0}(z^{q^{k_0}})=\bTheta_{k_0}(z^{q^{k_0}}) A_{k_0}(z)^{-1} A_{k_0}(z)$  
and using   \eqref{eq:P2-iterationF}, we get that 
\begin{equation}\label{eq:P2-nulliteF}
F \left( \frac{d(z^{q^{k_0}})q(z)}
{d(\alpha^{q^{k_0}})q(\alpha)}\bTheta_{k_0}(z^{q^{k_0}})
A_{k_0}(z)^{-1},z\right)=0 \,.
\end{equation}
Now, let  us consider the linear form in $X_1,\ldots,X_n$ defined by:
$$
Q(z,\X):=\btau \left( \frac{d(z^{q^{k_0}})q(z)}
{d(\alpha^{q^{k_0}})q(\alpha)}
\bTheta_{k_0}(z^{q^{k_0}}) A_{k_0}(z)^{-1}\right)\X \, .
$$ 
Thus, the coefficient of each $X_i$ in $Q(z,\X)$ belongs to $\Q[z,\varphi(z^{q^{k_0}})]$. 
Since $\varphi(z^{q^{k_0}})$ is analytic at $\alpha$, the coefficients of $Q(z,\X)$ are analytic at $\alpha$. Moreover, since $\bTheta_{k_0}(\alpha^{q^{k_0}})=\bTheta_{k_0}(\xi)= A_{k_0}(\alpha)$, 
we deduce that 
$$
Q(\alpha,\X)= \btau \X= L(\X)\, .
$$
Finally, it follows from  \eqref{eq:P2-nulliteF} that 
$$
Q(z,\f(z))= 0\,.
$$
There is only one point left to address: we have lifted the linear relation between $f_1(\alpha),\ldots,f_m(\alpha)$ into a linear relation between $f_1(z),\ldots,f_m(z)$, but this relation is over $\Q[z,\varphi(z^{q^{k_0}})]$. 
Since the field  $\Q(z,\f(z))$ is a regular extension of $\Q(z)$ (see \cite[Lemme 3.2]{AF17}), we have that $\Q(z)(\f(z))$ and $\mathbb A$ 
are linearly disjoint over $\Q(z)$ (see \cite[Chapter VIII]{Lang}).   
Let $\delta$ denote the degree of $\varphi(z^{q^{k_0}})$ over $\Q(z)$, so that the functions 
$\varphi(z^{q^{k_0}})^j$, 
$0\leq j  \leq \delta-1$, are linearly independent over $\Q(z)$. Since $\Q(z)(\f(z))$ and $\mathbb A$ 
are linearly disjoint over $\Q(z)$, these functions remain linearly independent over 
$\Q(z)(\f(z))$. Thus, splitting the linear form $Q$ as  
$$
Q(z,\X)= :\sum_{j=0}^{\delta-1} Q_j(z,\X)\varphi(z^{q^{k_0}})^j\,,
$$
where $Q_{j}(z,\X)\in \Q[z,\X]$ are linear forms,
we deduce that  
$$Q_j(z,\f(z))=0\,, \quad\quad \forall j, \;0\leq j\leq \delta-1\,.$$ 
Finally, setting 
$$
\overline{L}(z,\X):= \sum_{j=0}^{\delta-1} Q_{j}(z,\X)\varphi(\alpha^{q^{k_0}})^j\in \Q[z,\X]\,,
$$
we obtain that 
$
\overline{L}(z,\f(z))=0$  and  $\overline{L}(\alpha,\X)=L(\X)$,  
as wanted. This ends the proof of Theorem \ref{thm:P2-liftinglinear}.
\qed

\section{From linear to algebraic relations}\label{sec:P2-Kroneck}

In this section, we end the proof of Theorem \ref{thm:P2-Philippon}. 
In order to deduce Theorem  \ref{thm:P2-Philippon} from Theorem \ref{thm:P2-liftinglinear}, 
the key observation is that, given $M_q$-functions $f_1(z),\ldots,f_m(z)$ related by a $q$-Mahler system, the $M_q$-functions obtained by considering all monomials of a given degree in $f_1(z),\ldots,f_m(z)$ are also related by a $q$-Mahler system with no additional singularity. 

Let us first recall some notation. Let $A=(a_{i,j})$ and $B$ be matrices with entries in a given commutative ring, with dimension, respectively, $(m,n)$ and $(p,q)$. The Kronecker product of $A$ and $B$ is the matrix $A \otimes B$, of size $(mp,nq)$ with block decomposition
$$
A \otimes B := \left(\begin{array}{ccc} a_{1,1}B & \cdots & a_{1,n}B \\ \vdots & \ddots & \vdots \\ a_{m,1}B & \cdots & a_{m,n}B \end{array}\right) \, .
$$
If $d \geq 1$ is an integer, we also let
$$
A^{\otimes d} := \underbrace{A \otimes \cdots \otimes A}_{d \text{ times }}\, ,
$$
denote the $d$th Kronecker power of the matrix $A$.

\begin{proof}[Proof of Theorem \ref{thm:P2-Philippon}] 
Let $d$ denote the total degree of $P$  and $\lambd_1,\ldots,\lambd_t$ be an enumeration of the set 
$\{\lambd \in (\Z_{\geq 0})^m\, : \, \vert \lambd \vert = d\}$. Then, we have 
$$
P =: \sum_{j=1}^t p_j \X^{\lambd_j}\, ,
$$
where $p_j \in \Q$ and  $\X:=(X_1,\ldots, X_m)$. 
Set $\f(z):=(f_1(z),\ldots,f_m(z))^\top$. The coordinates of the vector $\f(z)^{\otimes d}$ are precisely the monomials of degree $d$ in $f_1(z),\ldots,f_m(z)$, with some of them appearing several times (for example, the product $f_1(z)f_2(z)$ appears twice in $\f(z)^{\otimes 2}$). 
Using \cite[Lemma 4.2.10]{HJ94} or (i) of Lemma \ref{lem:P2-proprieteKronecker} in Section  \ref{app:kronecker} and a straightforward induction on $d$, we obtain that
\begin{equation}\label{eq: sys-kronecker}
\f(z)^{\otimes d} = A(z)^{\otimes d}\f(z^q)^{\otimes d}\, .
\end{equation}
Since $\alpha$ is a regular point with respect to the system \eqref{eq:P2-MahlerSystem} the matrix $A(z)$ is well-defined and invertible at $\alpha^{q^k}$ for all integers $k \geq 0$. The entries of the matrix $A(z)^{\otimes d}$ being products of the entries of $A(z)$, the matrix $A(z)^{\otimes d}$ is well-defined at $\alpha^{q^k}$ for all integers $k \geq 0$. Furthermore, since $\det A(\alpha^{q^k})\neq 0$ we have $\det A(\alpha^{q^k})^{\otimes d}  \neq 0$ (see \cite[Corollary  4.2.11]{HJ94} or (ii) of Lemma \ref{lem:P2-proprieteKronecker} in Section  \ref{app:kronecker}), for all integers $k \geq 0$. Hence $\alpha$ is a regular point with respect to the system \eqref{eq: sys-kronecker}.

For each $j$, $1\leq j \leq t$, let $\I_j \subset \{1,\ldots,m^d\}$ denote the set of integers $i$ for which the $i$th entry of $\X^{\otimes d}$ is $\X^{\lambd_j}$. For each $j$, we pick an integer $i_j$  in $\I_j$. 
Let $Y_1,\ldots,Y_{m^d}$ be a family of indeterminates and let us consider the linear form $L$ defined by  
$$
L(Y_1,\ldots,Y_{m^d}):= \sum_{j=1}^t p_j Y_{i_j}\, .
$$ 
We also let $g_1,\ldots,g_{m^d}$ denote the coordinates of $\f(z)^{\otimes d}$. By construction $g_{i}(z)=\f(z)^{\lambd_j}$ when $i \in \I_j$. 
 Thus,
$$
L(g_{1}(\alpha),\ldots,g_{m^d}(\alpha))=\sum_{j=1}^t p_jg_{i_j}(\alpha)=P(f_1(\alpha),\ldots,f_m(\alpha))=0\,.
$$
By Theorem \ref{thm:P2-liftinglinear}, there exists $\overline{L}\in\Q[z,Y_1,\ldots,Y_{m^d}]$ linear in $Y_1,\ldots,Y_{m^d}$, such that
\begin{equation*}\label{eq: Lbar}
\overline{L}(z,g_{1}(z),\ldots,g_{m^d}(z))=0 \quad \text{ and } \quad \overline{L}(\alpha,Y_1,\ldots,Y_{m^d})=L(Y_1,\ldots,Y_{m^d})\, .
\end{equation*}
Write $\overline{L}=:\sum_{i=1}^{m^d} l_i(z)Y_i$, where $l_0(z),\ldots,l_{m^d}(z) \in \Q[z]$. 
We deduce  that 
$$
l_i(\alpha)= \left\{\begin{array}{cl}
p_j & \text{ if } i = i_j \text{ for some } j,\, 1\leq j \leq t\,,
\\ 0 & \text{ otherwise}\,.
\end{array}\right.
$$
Now, set
$$
\overline{P}(z,X_1,\ldots,X_m) := \sum_{j=1}^t \left(\sum_{i \in \I_j} l_i(z) \right)\X^{\lambd_j}\, .
$$
On the one hand, we have
\begin{eqnarray*}
\overline{P}(z,f_1(z),\ldots,f_m(z)) &=&  \sum_{j=1}^t \left(\sum_{i \in \I_j} l_i(z) \right)\f(z)^{\lambd_j}
\\ & = & \sum_{i=1}^{m^d} l_i(z)g_i(z)
\\ & = & \overline{L}(z,g_{1}(z),\ldots,g_{m^d}(z))=0 \, ,
\end{eqnarray*} 
while, on the other hand, we have 
$$
	\overline{P}(\alpha,\X)= \sum_{j=1}^t \left(\sum_{i \in \I_j} l_i(\alpha) \right)\X^{\lambd_j}
	= \sum_{j=1}^t p_j \X^{\lambd_j}
=P(\X)\, .
$$
This ends the proof. 
\end{proof}

\section{Deducing Nishioka's theorem from the lifting theorem}\label{sec: hilbert}

In this section, we show how to deduce Nishioka's theorem from the lifting theorem. 

\begin{proof}[Proof of Theorem \ref{thm:P2-Nishioka}]
We first note that the inequality
   $$
   {\rm tr.deg}_{\Q}(f_1(\alpha),\ldots,f_m(\alpha)) \leq {\rm
   tr.deg}_{\Q(z)}(f_1(z),\ldots,f_m(z))
   $$
   always holds. Hence we only have to prove that
   \begin{equation}
   \label{ineg:P2-degtrans}
  t_\alpha:= {\rm tr.deg}_{\Q}(f_1(\alpha),\ldots,f_m(\alpha)) \geq {\rm
   tr.deg}_{\Q(z)}(f_1(z),\ldots,f_m(z))=:t_z\, .
   \end{equation}
    Let $d\geq 0$ be an integer. We let
   $\varphi_\alpha(d)$ denote the dimension of the $\Q$-vector space spanned by the 
   monomials of degree at most $d$ in $f_1(\alpha),\ldots,f_m(\alpha)$. 
   We also let $\varphi_z(d)$ denote the dimension of the $\Q(z)$-vector space 
   spanned by the monomials of degree at most $d$ in 
   $f_1(z),\ldots,f_m(z)$. 
   Note that the functions $1,f_1(z),\ldots,f_m(z)$ are related by the $q$-Mahler system of size $m+1$:
\begin{equation}\label{eq: systinho}
\left(\begin{array}{c}1\\ f_1(z) \\ \vdots \\ f_m(z)\end{array}\right)
=
\left(\begin{array}{c|ccc}1&&& \\ \hline &&& \\ &&A(z)&\\&&&  \end{array}\right)
\left(\begin{array}{c}1\\f_1(z^q) \\ \vdots \\ f_m(z^q)\end{array}\right) \,.
\end{equation}
Furthermore, the point $\alpha$ remains regular with respect to this new system. 
   Applying Theorem \ref{thm:P2-Philippon} to \eqref{eq: systinho}, we obtain that
   \begin{equation}\label{eq:P2-hilbertserre}
   \varphi_\alpha(d) \geq \varphi_z(d),\, \quad \forall d\geq 0\, .
   \end{equation}
    By a result of Hilbert,  $\varphi_\alpha(d)$ and  $\varphi_z(d)$ are polynomials in $d$ of degree res\-pectively equal to $ t_\alpha$  and $t_z$ when $d\gg 1$ (see, for instance, the discussion around the Hilbert-Serre theorem in \cite[p.\ 232]{ZS}). Thus, there exist two positive real numbers $\beta$ and $\gamma$ such that 
$$
\varphi_\alpha(d) \leq \beta d^{t_\alpha} \ \text{ and } \ \varphi_z(d) \geq \gamma d^{t_z}\, ,\quad \quad \forall d \gg 1 \,.
$$ 
Using \eqref{eq:P2-hilbertserre}, we deduce \eqref{ineg:P2-degtrans}  as wanted. 
\end{proof}

\begin{rem}\label{rem:Hilbert}
In the proof of Theorem \ref{thm:P2-Nishioka}, we do not need the full strength of Hilbert's result. Suitable estimates  for $\varphi_\alpha(d)$ and  
$\varphi_z(d)$ can be easily achieved by elementary means (see Section \ref{app:Hilbert}).
\end{rem}

\begin{rem}
At the end of our proof of Theorem~\ref{thm:P2-Philippon}, we used the fact that the field extension $\Q(z,f_1(z),\ldots,f_m(z))$ is a regular extension of $\Q(z)$. We stress that this argument is not needed to deduce Nishioka's theorem. Indeed, 
without using it, we still obtain that every $\Q$-linear relation between $f_1(\alpha),\ldots,f_m(\alpha) $ can be lifted into a linear relation over the algebraic closure $\mathbb A$ of $\Q(z)$ between $f_1(z),\ldots,f_m(z)$. Then we could reproduce the previous argument, just replacing $\Q(z)$ by $\mathbb A$. We would derive the main result since 
$$
 {\rm tr.deg}_{\Q(z)}(f_1(z),\ldots,f_m(z))= {\rm tr.deg}_{\mathbb A}(f_1(z),\ldots,f_m(z)) \,,
$$
$\mathbb A$ being by definition algebraic over $\Q(z)$. 
\end{rem}


\appendix\label{App}

\section{Addendum to the proofs of Theorems \ref{thm:P2-Nishioka} and 
\ref{thm:P2-Philippon} }

In Sections \ref{app:dimension} and \ref{app:height}, we provide more details about some  auxiliary results used in the proof of Theorem \ref{thm:P2-liftinglinear}. 
We also give in Sections \ref{app:kronecker} and \ref{app:Hilbert} 
the proof of two elementary auxiliary results that can be used to make  the proof  of Theorems \ref{thm:P2-Nishioka} and 
\ref{thm:P2-Philippon} as elementary and self-contained as possible.

\subsection{Computation of the dimension of $\mathcal I^\perp(\delta_1,\delta_2)$}\label{app:dimension}
We provide here a more detailed argument for the proof of  Lemma \ref{lem:P2-dimensionespace}. 

\begin{proof}[Proof of Lemma \ref{lem:P2-dimensionespace}]

Set $h:=(\delta_1+1)^{m^2}$ and let $\bnu_1,\ldots,\bnu_h$ denote an enumeration of the set of all  
matrices in ${\mathcal M}_m(\Z_{\geq 0})$ whose entries are at most $\delta_1$. Any polynomial $P \in \Q(z)[\Y]_{\delta_1}$ has a unique decomposition of the form    
$$
P(\Y,z)=: \sum_{j=1}^h p_{j}(z) \Y^{\bnu_j}\, ,
$$
where $p_{j}(z) \in \Q(z)$, $1 \leq j \leq h$. 
Since, by definition, $\I(\delta_1)$ does not contain any non-zero elements of $\Q$, it is a strict $\Q(z)$-subspace of 
$\Q(z)[\Y]_{\delta_1}$. 
Thus, there exist an integer $d\geq 1$ and $d$ vectors of polynomials $(b_{i,1}(z),\ldots,b_{i,h}) \in \Q[z]^h$, $1\leq i \leq d$, which are linearly independent over $\Q(z)$ and such that  
for all $p_1(z),\ldots,p_h(z) \in \Q(z)$:
\begin{equation}\label{eq:App_P2-conditionI}
\sum_{j=1}^h p_j(z)\Y^{\bnu_j} \in \I(\delta_1) \Leftrightarrow \sum_{j=1}^h b_{i,j}(z)p_j(z) = 0 
\;\;\;\; \forall i,\; 1 \leq i \leq d\, .
\end{equation}
Since none of the vectors $(b_{i,1}(z),\ldots,b_{i,h}(z))$ is zero, we can choose the polynomials $b_{i,j}(z)$ 
so that for every $i$, $1\leq i\leq d$, there exists $j_i$, $1\leq j_i\leq h$, such that 
\begin{equation}\label{eq:App_ bij}
b_{i,j_i}(0)\not=0\,.
\end{equation} 
Since these polynomials only depend on $\delta_1$ (and $\mathcal I$), there exists a non-negative integer $\delta'_1$, which only depends on $\delta_1$ (and $\mathcal I)$, such that they can be written as 
$$
b_{i,j}(z)=:\sum_{\kappa =0}^{\delta_1'} b_{i,j,\kappa} z^\kappa\, ,
$$
where the numbers $b_{i,j,\kappa}$ belong to $\Q$. 
For every integer $\kappa$ such that $\kappa >\delta'_1$ or $\kappa<0$, we also set $b_{i,j,\kappa}:=0$.  
Now, let us consider a polynomial 
$P(\Y,z) =\sum_{j=1}^h p_j(z)\Y^{\bnu_j} \in \Q[\Y,z]_{\delta_1,\delta_2}$ and set  
$$
p_j(z) =: \sum_{\lambda =0}^{\delta_2} p_{j,\lambda} z^\lambda \, ,
$$
where the numbers $p_{j,\lambda}$ belong to $\Q$. We also set $p_{j,\lambda}:=0$ if $\lambda >\delta_2$ or $\lambda < 0$.  
By \eqref{eq:App_P2-conditionI}, $P$ belongs to $\I(\delta_1,\delta_2)$ if and only if 
\begin{equation}\label{eq:App_ gammarange}
\sum_{j=1}^h \sum_{ \kappa=0}^{\delta_1'} b_{i,j,\kappa}  p_{j,\gamma-\kappa} = 0 \, , \quad\quad \forall (\gamma,i)\,, \, 0\leq \gamma \leq  \delta_2+\delta'_1\,,\, 1\leq i\leq d\,.
\end{equation}
For each $\delta_2\geq\delta'_1$, we let $V(\delta_1,\delta_2)$ denote the $\Q$-vector subspace of the dual of $\Q[\Y,z]$ spanned 
by the linear forms
\begin{equation}\label{eq:App_ gammarange2}
L_{\gamma,i}:\sum_{j=1}^h\sum_{\lambda}p_{j,\lambda}z^\lambda\Y^{\bnu_j} \mapsto 
\sum_{j=1}^h\sum_{\kappa=0}^{\delta'_1}b_{i,j,\kappa}p_{j,\gamma-\kappa}\, ,\;\; \quad \delta'_1 \leq \gamma\leq \delta_2\,,  1\leq i\leq d\, .
\end{equation}
Since in \eqref{eq:App_ gammarange} and \eqref{eq:App_ gammarange2} the parameter $\gamma$ runs over the same interval up to a finite set that does not depend on $\delta_2$, it follows that 
\begin{equation}\label{eq:App_ comparaisonV}
\dim \I^\perp(\delta_1,\delta_2)= \dim V(\delta_1,\delta_2) + \mathcal O(1) \,,
\end{equation}
as $\delta_2$ tends to infinity. For every $\delta_2 \geq \delta_1'$, we set 
$$
c(\delta_1,\delta_2):=\dim V(\delta_1,\delta_2+1) - \dim V(\delta_1,\delta_2)\,.
$$ 
\subsubsection*{Frist claim.} We have $c(\delta_1,\delta_2)\geq 1$ for all $\delta_2\geq \delta_1$.

\begin{proof}[Proof of the first claim]
	Let us observe that, for every $(\gamma,\lambda,i,j)$, we have 
	\begin{equation}\label{eq:App_ Lgi}
	L_{\gamma,i}(z^\lambda\Y^{\bnu_j})=b_{i,j,\gamma-\lambda}\,.
	\end{equation}
	Hence $L_{\gamma,i}(z^{\delta_2+1}\Y^{\bnu_j})=0$ for all $\gamma\leq \delta_2$ and all $(i,j)$, while for every $i$, $1\leq i\leq d$, we infer from \eqref{eq:App_ Lgi} and \eqref{eq:App_ bij} that 
	$L_{\delta_2+1,i}(z^{\delta_2+1}\Y^{\bnu_{j_i}})=b_{i,j_i,0}=b_{i,j_i}(0)\not=0$. 
	We deduce that $L_{\delta_2+1,i}$ does not belong to $V(\delta_1,\delta_2)$ and hence $c(\delta_1,\delta_2)\geq 1$. 
\end{proof}

\subsubsection*{Second claim.} The sequence $(c(\delta_1,\delta_2))_{\delta_2\geq \delta'_1}$ is non-increasing. 

\begin{proof}[Proof of the second claim]
	Note that, by definition, $c(\delta_1,\delta_2)$ is equal to the number of linear forms in $\{L_{\delta_2+1,i} : 1\leq i\leq d\}$ which 
	are linearly independent over $V(\delta_1,\delta_2)$. Let us assume that some of these linear forms, say 
	$L_{\delta_2+1,i_1},\ldots,L_{\delta_2+1,i_t}$ are linearly dependent over $V(\delta_1,\delta_2)$, that is 
	$$
	\sum_{k=1}^t\theta_kL_{\delta_2+1,i_k}\in V(\delta_1,\delta_2)\,,
	$$
	where $\theta_1,\ldots,\theta_t$ are algebraic numbers, not all zero. Then we are going to show that we also have 
	\begin{equation}\label{eq:App_ claim2}
	\sum_{k=1}^t\theta_kL_{\delta_2+2,i_k}\in V(\delta_1,\delta_2+1)\,.
	\end{equation}
	The second claim follows directly from \eqref{eq:App_ claim2}. 
	
	By assumption, we can write 
	$$
	\sum_{k=1}^t\theta_kL_{\delta_2+1,i_k} = \sum_{\substack{\delta'_1\leq \gamma \leq \delta_2 \\1\leq i\leq d}} 
	\eta_{\gamma,i} L_{\gamma,i} \,,
	$$
	where the $\eta_{\gamma,i}$ are algebraic numbers. 
	In order to prove \eqref{eq:App_ claim2}, we show more precisely that 
	$$
	\sum_{k=1}^t\theta_kL_{\delta_2+2,i_k} = \sum_{\substack{\delta'_1\leq \gamma  \leq \delta_2\\ 1\leq i\leq d}} 
	\eta_{\gamma,i} L_{\gamma+1,i} 
	$$
	or, equivalently, that 
	$$
	\sum_{k=1}^t\theta_kL_{\delta_2+2,i_k}(z^\lambda\Y^{\bnu_j})=\sum_{\substack{\delta'_1\leq \gamma  \leq \delta_2\\ 1\leq i\leq d}} 
	\eta_{\gamma,i} L_{\gamma+1,i}(z^\lambda\Y^{\bnu_j})\,,
	$$
	for all $(\lambda,j)$, $\lambda\geq 0$, $1\leq j\leq h$. 
	If $\lambda>0$, we have 
	\begin{eqnarray*}
		\sum_{k=1}^t\theta_kL_{\delta_2+2,i_k}(z^\lambda\Y^{\bnu_j}) &= &\sum_{k=1}^t \theta_k b_{i_k,j,\delta_2+2-\lambda}\\
		&=& \sum_{k=1}^t\theta_kL_{\delta_2+1,i_k}(z^{\lambda-1}\Y^{\bnu_j}) \\
		&=&    \sum_{\substack{\delta'_1\leq \gamma  \leq \delta_2\\ 1\leq i\leq d}} 
		\eta_{\gamma,i} L_{\gamma,i}(z^{\lambda-1}\Y^{\bnu_j}) \\
		&=&   \sum_{\substack{\delta'_1\leq \gamma  \leq \delta_2\\ 1\leq i\leq d}} \eta_{\gamma,i}  b_{i,j,\gamma+1-\lambda}\\
		&=&   \sum_{\substack{\delta'_1\leq \gamma  \leq \delta_2\\ 1\leq i\leq d}} \eta_{\gamma,i}  L_{\gamma + 1,i}(z^{\lambda}\Y^{\bnu_j})\,.
	\end{eqnarray*}
	If $\lambda=0$, we have 
	$$
	\sum_{k=1}^t\theta_kL_{\delta_2+2,i_k}(\Y^{\bnu_j})= \sum_{k=1}^t \theta_k  b_{i_k,j,\delta_2+2}=0 
	$$
	and
	$$
	\sum_{\substack{\delta'_1\leq \gamma  \leq  \delta_2\\ 1\leq i\leq d}} 
	\eta_{\gamma,i} L_{\gamma+1,i}(\Y^{\bnu_j}) =  \sum_{\substack{\delta'_1\leq \gamma  \leq \delta_2\\ 1\leq i\leq d}} 
	\eta_{\gamma,i} b_{i,j,\gamma+1}=0 \,.
	$$
	This ends the proof of the second claim. 
\end{proof}

\subsubsection*{End of the proof  of Lemma \ref{lem:P2-dimensionespace}} 
It follows from the two claims  that the sequence $(c(\delta_1,\delta_2))_{\delta_2\geq \delta'_1}$ is eventually constant. Let $c(\delta_1)\geq 1$ denote its limit. Then $c(\delta_1)$ is an integer which does not depend on $\delta_2$ and we have 
$$
\dim V(\delta_1,\delta_2) \sim c(\delta_1)\delta_2  \,, \mbox{ as $\delta_2$ tends to infinity. }
$$
By \eqref{eq:App_ comparaisonV}, we deduce that
$$
\dim \I^{\perp}(\delta_1,\delta_2) \sim c(\delta_1)\delta_2  \,, \mbox{ as $\delta_2$ tends to infinity,}
$$
as wanted. 
\end{proof}

\subsection{Height computations}\label{app:height}

In Section \ref{subsec:klemma},  we obtained an upper bound for the logarithmic Weil height of the algebraic numbers $P_{v_0}(A_k(\alpha),\alpha^{q^k})$ during Step (LB). 
To that end, we claimed that the height of the coordinates of $A_k(\alpha)$ belong to $\mathcal O(q^k)$ as $k$ tends to infinity. We provide here a proof of this claim. 

We first recall the definition of the (absolute) logarithmic Weil height of an algebraic number.  
Let $\bf k$ be a number field. The absolute logarithmic height of a projective point 
$(\alpha_0:\cdots :\alpha_n)\in \mathbb P_n(\bf k)$ is defined by 
$$
h(\alpha_0:\cdots :\alpha_n) := \frac{1}{[\bf k:\mathbb Q]}\sum_{\nu \in M_{\bf k}}d_{\nu}\log \max 
\{\vert \alpha_0\vert_{\nu},\ldots,\vert \alpha_n\vert_{\nu}\} \, ,
$$
where $\nu$ runs over a complete set $M_{\bf k}$ of non-equivalent places of $\bf k$, 
$d_{\nu}=[{\bf k}_{\nu}:\mathbb Q_p]$\footnote{Here we let ${\bf k}_{\nu}$ 
	denote the completion of ${\bf k}$ with respect to $\nu$ and $\nu\vert_{\mathbb Q}=p$, with the convention that if 
	$\nu\vert_{\mathbb Q}=\infty$ then $\mathbb Q_p=\mathbb R$ and ${\bf k}_{\nu}$ is either 
	$\mathbb R$ if the place is real or $\mathbb C$ if the place is complex.},  and where the absolute values 
$\vert\cdot\vert_{\nu}$ are normalized so that the product formula holds: 
$$
\prod_{\nu \in M_{\bf k}} \vert x\vert_{\nu}^{d_{\nu}} = 1\,, \quad\quad \forall x\not=0\in {\bf k} \, .
$$
We also set $h(\alpha):=h(1:\alpha)$ for all $\alpha\in\bf k$, so that $h(0)=0$. 
We recall that the height of an algebraic number $\alpha$ 
does not depend on the choice of the number field 
${\bf k}$ containing it. Let $\mathcal S$ be a finite set of places of ${\bf k}$ containing the Archimedean ones -- that is the places such that
$\nu_{\vert \mathbb Q} = \infty$. 
An element $\alpha \in {\bf k}$ such that $\vert \alpha \vert_\nu \leq 1$ 
for every $\nu \in M_{\bf k}\setminus \mathcal S$ is called an \textit{$\mathcal S$-integer}. We recall that $\mathcal S$-integers form a ring and that for every $\mathcal S$-integer $\alpha$ one has
\begin{equation}\label{eq: heightS}
h(\alpha)=\frac{1}{[\bf k:\mathbb Q]}\sum_{\nu \in \mathcal S}d_{\nu}\log \max 
\{1,\vert \alpha \vert_{\nu}\} \,.
\end{equation}

For every $k\geq 0$, we set 
$A_k(z)=:(a_{i,j}^{(k)}(z))_{1\leq i,j\leq m}$, so that $A(z)=(a_{i,j}^{(1)}(z))_{1\leq i,j\leq m}$.

\begin{lem}
	\label{lem: hauteurAk}
	Let $\bf k$ denote the number field generated over $\mathbb Q$ by $\alpha$ and the coefficients of the rational functions that form the coordinates of  $A(z)$. Then there  exists a positive real number $\gamma$ such that for every integers $k$, $i$, and $j$, $k\geq 0$, $1\leq i,j\leq m$, $a_{i,j}^{(k)}(\alpha)$ belongs to $\bf k$ and 
	$h(a_{i,j}^{(k)}(\alpha))\leq \gamma q^k$. 
\end{lem}

\begin{proof}
	A straightforward induction using the relation 
	\begin{equation}\label{eq: recAk}
	A_{k+1}(z)=A_k(z)A(z^{q^k})
	\end{equation}
	shows that all the coefficients $a_{i,j}^{(k)}(\alpha)$ belong to $\bf k$. 
	
	We  first assume  that $A(z)$ has coefficients in $\Q[z]$. There exists a finite set of places $\mathcal S$ of $\bf k$ (containing the Archimedean  ones) such that 
	all the coefficients of all the coordinates of $A(z)$ and $\alpha$ are $\mathcal S$-integers. Then, it follows from \eqref{eq: recAk} 
	that all the numbers $a_{i,j}^{(k)}(\alpha)$ are also $\mathcal S$-integers (since the latter form a ring).
	Let $\nu$ be an element of $\mathcal S$.  
	There exists a positive real number $\gamma_1(\nu)\geq 1$ that does not depend on $k$ such that 
	\begin{equation*}
	\vert a_{i,j}^{(1)}(\alpha^{q^k})\vert_{\nu} \leq e^{\gamma_1(\nu) q^k} \,, \quad\quad \forall k\geq 0\,, 1\leq i,j\leq m\,.
	\end{equation*}
	Set $\gamma_1:=\max\{\gamma_1(\nu) : \nu \in \mathcal S\}$. 
	
	Let us prove by induction on $k$ that for all $\nu\in\mathcal S$: 
	\begin{equation}\label{eq: recak}
	\vert a_{i,j}^{(k)}(\alpha) \vert_\nu \leq e^{\gamma_2q^k} \,, \quad\quad \forall k\geq 0, 1\leq i,j\leq m\,, 
	\end{equation}
	where $\gamma_2:= \gamma_1 +\log m$. 
	The result is trivial for $k=0$, since $A_0(z)={\rm I}_m$. 
	Let us assume that the result holds for some $k\geq 0$. 
	We infer from \eqref{eq: recAk} that 
	$$
	a_{i,j}^{(k+1)}(\alpha)= \sum_{\ell=1}^ma_{i,\ell}^{(k)}(\alpha)a_{\ell,j}^{(1)}(\alpha^{q^k}) \,.
	$$
	It follows that 
	\begin{eqnarray*}
		\vert a_{i,j}^{(k+1)}(\alpha)\vert_{\nu} &\leq& \sum_{\ell=1}^m\vert a_{i,\ell}^{(k)}(\alpha)\vert_\nu\,\vert a_{\ell,j}^{(1)}(\alpha^{q^k})\vert_{\nu}\\
		&\leq & me^{(\gamma_1+\gamma_2)q^k}\\
		&\leq &e^{\gamma_2q^{k+1}}\,.
	\end{eqnarray*}
	This proves \eqref{eq: recak}. We deduce from \eqref{eq: heightS} that 
	\begin{eqnarray*}
		h(a_{i,j}^{(k)}(\alpha))&=& \frac{1}{[\bf k: \mathbb Q]}\sum_{\nu\in\mathcal S}d_\nu\log \max\{1, \vert a_{i,j}^{(k)}(\alpha)\vert_{\nu} \}\\
		&\leq & \gamma_3 q^k \,,
	\end{eqnarray*}
	where $\gamma_3:= \frac{\gamma_2}{[\bf k: \mathbb Q]} \sum_{\nu\in\mathcal S}d_\nu$. Taking $\gamma:=\gamma_3$, this proves the lemma when $A(z)$ has polynomial coefficients.
	
	We prove now the general case where the coordinates of $A(z)$ are rational functions.
		Let $b(z) \in\Q[z]$ denote the least common multiple 
	of the denominators of the coordinates of $A(z)$. Hence the matrix $b(z)A(z)$ 
	has coefficients in $\Q[z]$. For every $k \geq 0$, we set 
	\begin{equation}\label{eq: bk}
	b_k(z):=b(z)b(z^q)\cdots b(z^{q^{k-1}})\, ,
	\end{equation}
	so that $b_0(z)=1$, $b_1(z)=b(z)$, and the matrix $b_k(z)A_k(z)$ has coefficients in $\Q[z]$. 
	It follows from the first part of the proof that the height of the coordinates of the matrix $b_k(\alpha)A_k(\alpha)$ is at most $\gamma_4 q^k$, for some positive real 
	number $\gamma_4$ which does not depend on $k$.  Using \eqref{eq: height} and \eqref{eq: height3},  we obtain that $h(b(\alpha^{q^k})) = \mathcal O(q^k)$ as $k$ tends to infinity.  By \eqref{eq: height}, we deduce that there exists a real number $\gamma_5>0$ such that
	$$
	h(b_k(\alpha)) \leq \gamma_5q^k, \quad\quad  \forall k\geq 0\,.
	$$
	Finally,  \eqref{eq: height} implies that 
	$$
	h(a_{i,j}^{(k)}(\alpha)) \leq (\gamma_4+\gamma_5)q^k
	$$
	and the lemma holds with $\gamma:=\gamma_4+\gamma_5$.
\end{proof}

\subsection{Kronecker product}\label{app:kronecker}

For the sake of completeness, we provide a proof of the basic properties that we used about Kronecker product in Section \ref{sec:P2-Kroneck}.

\begin{lem}\label{lem:P2-proprieteKronecker} The following properties hold. 
	
	\medskip 
	\begin{enumerate}[label=\roman*)]
		\item[{\rm (i)}]    If $A$,$B$, $C$, and $D$ are matrices such that the products $AB$ and 
		$CD$ are well-defined, then  
		$
		(AB)\otimes(CD)=(A \otimes C)(B \otimes D)$. 
		
		\medskip
		
		\item[{\rm (ii)}] If $A$ and $B$ are two square matrices respectively of size $m$ and $n$, then 
		$$\det(A\otimes B)=(\det A)^n(\det B)^m\,.$$ In particular, if $A$ is a square matrix of size $m$ and $d\geq 1$ is an integer, then
		$$
		\det A^{\otimes d} = (\det A)^{md}\, .
		$$
	\end{enumerate}
\end{lem}

\begin{proof}
	Let us first prove (i). Set $A:=(a_{i,j})_{i,j}$ and $B:=(b_{j,k})_{j,k}$. Then we have the block matrix decompositions
	$$
	A\otimes C = (a_{i,j}C)_{i,j}, \ \text{ and } \ B\otimes D = (b_{j,k}D)_{j,k}\,.
	$$
	According to this block decomposition, the $(i,k)$th block of $(A \otimes C)(B \otimes D)$ is
	$$
	\sum_j a_{i,j}b_{j,k}CD\, .
	$$
	In the mean time, $AB=(\sum_j a_{i,j}b_{j,k})_{i,k}$ and thus, the $(i,k)$th block of $(AB)\otimes(CD)$ is precisely 
	$\sum_j a_{i,j}b_{j,k}CD$.
	
	Now, let us prove (ii). 
	Consider the matrix $A \otimes {\rm I}_n$. After a suitable permutation of the rows and the columns, 
	one obtains the matrix ${\rm I}_n \otimes A$, which is a block diagonal matrix. Hence
	$$
	\det(A \otimes {\rm I}_n)=\det({\rm I}_n \otimes A) = (\det A)^n\,.
	$$
	Similarly, $\det({\rm I}_m \otimes B)= (\det B)^m$. We infer from (i) that  
	$$
	A\otimes B= (A{\rm I}_m)\otimes({\rm I}_nB)=(A \otimes {\rm I}_n)({\rm I}_m \otimes B)\, .
	$$
	Hence 
	\begin{equation*}
	\det (A\otimes B) = \det({\rm I}_n \otimes A)\det({\rm I}_m \otimes B)=(\det A)^n(\det B)^m
	\end{equation*}
	as wanted. The last property follows now  by induction on $d$. 
\end{proof}

\subsection{Hilbert's theorem about transcendence degree}\label{app:Hilbert}

In Section \ref{sec:P2-Kroneck}, 
we use a result of Hilbert to deduce Theorem \ref{thm:P2-Nishioka} from Theorem \ref{thm:P2-Philippon}.  
Let $\mathbb K$ be a field, $\mathbb L$ be a field extension of $\mathbb K$, $\xi_1,\ldots,\xi_m \in \mathbb L$, and $\varphi(d)$ denote the dimension of the 
$\mathbb K$-vector space formed by the polynomials in $\xi_1,\ldots,\xi_m$ of total degree at most $d$. Hilbert's theorem states that $\varphi(d)$ is a polynomial with degree 
$t:={\rm tr.deg}_{\mathbb K}(\xi_1,\ldots,\xi_m)$ for all $d$ large enough.  
As mentionned in Remark \ref{rem:Hilbert}, we do not need the full strength of Hilbert's theorem.  
The following elementary lemma is indeed sufficient to deduce Theorem \ref{thm:P2-Nishioka}.

\begin{lem}\label{lem:P2-HilbertSerre}
	We continue with the notation above. There exist two positive real numbers $\gamma_1$ and $\gamma_2$ that does not depend on $d$ and such that
	$$
	\gamma_1 d^{t} \leq \varphi(d) \leq \gamma_2 d^t,\quad \forall d\geq 1 \,.
	$$
\end{lem}

\begin{proof} Suppose first that $t=0$. Then, $\mathbb K(\xi_1,\ldots,\xi_m)$ is a finitely generated extension of $\mathbb K$ which is algebraic over $\mathbb K$. It follows that it has finite degree, say $\delta$, over $\mathbb K$. Then, for any $d$,
	$$
	1 \leq \varphi(d) \leq \delta\,.
	$$
	Taking $\gamma_1=1$ and $\gamma_2=\delta$, this proves the lemma when $t=0$. Suppose now that $t \geq 1$. Without any loss of generality, we can assume that $\xi_1,\ldots,\xi_t$ are algebraically independent over $\mathbb K$. 
	
	We first prove the lower bound. By assumption, all the monomials in $\xi_1,\ldots,\xi_t$ are linearly independent over $\mathbb K$. 
	Since there are $\binom{d+t}{t}$ distinct monomials in $\xi_1,\ldots,\xi_t$ of degree at most $d$, we obtain that 
	$$
	\varphi(d) \geq \binom{d+t}{t} \geq \frac{d^t}{t!}\,\cdot
	$$
	Taking $\gamma_1 = (t!)^{-1}$, we obtain the expected lower bound. 
	
	Now,  let us prove the upper bound by induction on $m\geq t$, the integer $t$ being fixed. Suppose that $m=t$, then
	$$
	\varphi(d)=\binom{d+t}{t} \leq (t+1)d^t\, .
	$$
	We assume now that $m>t$ and that the upper bound holds when considering $m-1$ elements in $\mathbb L$. 
	By assumption, $\xi_m$ is algebraic over $\mathbb K(\xi_1,\ldots,\xi_{m-1})$. 
	We first prove the case where $\xi_m$ is integer over $\mathbb K[\xi_1,\ldots,\xi_{m-1}]$. In that case, there exist an integer $\delta \geq 1$ and polynomials $P_{0},\ldots,P_{\delta-1} \in \mathbb K[X_1,\ldots,X_{m-1}]$ such that
	$$
	\xi_m^\delta =\sum_{i=0}^{\delta - 1} P_i(\xi_1,\ldots,\xi_{m-1})\xi_m^i\, .
	$$
	Let $d_0$ denote the maximum of the degree of $P_0,\ldots,P_{\delta-1}$. Let $d \geq \delta$, then any polynomial of 
	degree at most  $d$ in $\xi_1,\ldots,\xi_m$ can be decomposed as
	$$
	Q_0(\xi_1,\ldots,\xi_{m-1})+Q_1(\xi_1,\ldots,\xi_{m-1})\xi_m+\cdots+Q_{\delta-1}(\xi_1,\ldots,\xi_{m-1})\xi_m^{\delta-1}\, ,
	$$
	where $Q_1,\ldots,Q_{\delta-1}$ are polynomials of degree at most $(1+d_0)d$. By induction, there exists a positive number 
	$\gamma$ such that there are at most 
	$$
	\gamma ((1+d_0)d)^t
	$$
	linearly independent polynomials in $\xi_1,\ldots,\xi_{m-1}$ whose degree is at most $(1+d_0)d$. 
	Setting $\gamma_2:=\delta \gamma (1+d_0)^t$, we obtain the expected upper bound when 
	$\xi_m$ is integer over $\mathbb K[\xi_1,\ldots,\xi_{m-1}]$. 
	
	In the general case, there exists a polynomial $p \in \mathbb K[X_1,\ldots,X_{m-1}]$ such that $p(\xi_1,\ldots,\xi_{m-1})\xi_m$ is 
	integer over $\mathbb K[\xi_1,\ldots,\xi_{m-1}]$. Let $d_1$ denote the degree of $p$. Then the multiplication by $p(\xi_1,\ldots,\xi_{m-1})^d$ induces an injection between the vector space of polynomials of degree at most $d$ in $\xi_1,\ldots,\xi_m$ and the vector space of polynomials of degree at most $d(d_1+1)$ in $\xi_1,\ldots,\xi_{m-1},p(\xi_1,\ldots,\xi_{m-1})\xi_m$. 
	The dimension of the latter is at most
	$$
	\gamma(d_1+1)^t d^t 
	$$
	for some $\gamma>0$, since $p(\xi_1,\ldots,\xi_{m-1})\xi_m$ is integer over $\mathbb K[\xi_1,\ldots,\xi_{m-1}]$. 
	Setting $\gamma_2:= \gamma(d_1+1)^t$, we obtain the expected upper bound. 
\end{proof}



\begin{thebibliography}{99}




 \bibitem{AF17} B. Adamczewski et C. Faverjon, {\it M\'ethode de Mahler: relations lin\'eaires, transcendance et applications aux nombres automatiques}, 
 Proc. London Math. Soc. {\bf 115} (2017), 55--90.  
 
  \bibitem{AF18} B. Adamczewski et C. Faverjon, {\it M\'ethode de Mahler, transcendance et relations lin\'eaires: aspects effectifs}, 
J. Th\'eor. Nombres Bordeaux  {\bf 30} (2018), 557--573.  

 
  \bibitem{AF20} B. Adamczewski and C. Faverjon, {\it Mahler's method in several variables and finite automata}, preprint 2020, arXiv:2012.08283 [math.NT], 52 pp.
  
   \bibitem{AF22} B. Adamczewski et C. Faverjon, {\it A  new proof of Nishioka's theorem in Mahler's method}, 
 preprint 2022, arXiv: [math.NT], 27 pp. 
 
 \bibitem{An1} Y. Andr\'e, {\it S\'eries Gevrey de type arithm\'etique I. Th\'eor\`emes de puret\'e et de dualit\'e}, 
 Annals of Math. {\bf 151} (2000), 705--740.   

 \bibitem{An2} Y. Andr\'e, {\it S\'eries Gevrey de type arithm\'etique II. Transcendance sans transcendance}, 
 Annals of Math. {\bf 151} (2000),  
741--756.  


\bibitem{An3} Y. Andr\'e, {\it Solution algebras of differential equations and quasi-homogeneous varieties$:$ a new differential Galois correspondence},  
Ann. Sci. \'Ec. Norm. Sup\'er. {\bf 47}  (2014), 449--467.
 




\bibitem{Be06} F. Beukers, {\it A refined version of the Siegel--Shidlovskii theorem}, Annals of Math. 
{\bf 163} (2006), 369--379. 

\bibitem{Du} P. Dumas, R\'ecurrences mahl\'eriennes, suites automatiques, \'etudes asymptotiques, 
 Thèse de doctorat, Université de Bordeaux I, Talence, 1993.




\bibitem{FN} N. I. Fel{'}dman and Yu.  V. Nesterenko, 
Transcendental numbers. Number theory IV,  
\emph{Encyclopaedia Math. Sci.} {\bf 44}, Springer, Berlin, 1998. 

\bibitem{Fe18} G. Fernandes, \emph{M\'ethode de Mahler en caract\'eristique non nulle: un analogue du th\'eor\`eme de Ku. Nishioka}, Ann. Inst. Fourier (Grenoble) {\bf 68} (2018), 2553--2580.

\bibitem{HJ94} R. A. Horn and C. R. Johnson, Topics in matrix analysis, Cambridge University Press, Cambridge, 1994.

\bibitem{Ku77} K. K. Kubota, {\it On the algebraic independence of holomorphic solutions of certain functional equations and their values}, 
Math. Ann. {\bf 227} (1977), 9--50. 

\bibitem{Lang} S. Lang, Algebra, revised third edition, \emph{Graduate Texts in Mathematics}  {\bf 211}, Springer-Verlag, New York, 2002.

\bibitem{LvdP82} 
J. H. Loxton and A. J. van der Poorten, {\it Arithmetic properties of the solutions of a class of
 functional equations}, J. reine angew. Math. {\bf 330} (1982), 159--172.



\bibitem{Ma30b}
K. Mahler, 
{\it Arithmetische Eigenschaften einer Klasse transzendental-transzendente Funktionen}, Math. Z. {\bf 32} (1930), 545--585. 

\bibitem{NS20} L. Naguy and T. Szamuely, \emph{A general theory of Andr\'e's solution algebras}, Ann. Inst. Fourier (Grenoble) {\bf 70} (2020), 
2003--2129.

\bibitem{Ne77} Yu. V. Nesterenko, 
\emph{Estimate of the orders of the zeroes of functions of a certain class, and their application in the theory of transcendental numbers} (Russian),  Izv. Akad. Nauk SSSR Ser. Mat. {\bf 41} (1977), 253--284, 477. 


\bibitem{Ni90}
Ku. Nishioka, 
{\it New approach in Mahler's method}, 
J. reine angew. Math. {\bf 407} (1990), 202--219.

\bibitem{Ni94} Ku. Nishioka, \emph{Algebraic independence by Mahler's method and S-unit equations},  
Compos. Math. {\bf 92} (1994), 87--110.

\bibitem{Ni96} Ku. Nishioka, \emph{Algebraic independence of Mahler functions and their values}, Tohoku Math. J.  {\bf 48} (1996), 
51--70.


\bibitem{PPH} P. Philippon, {\it Groupes de Galois et nombres automatiques}, J. Lond. Math. Soc. {\bf 92} (2015),  596--614. 

\bibitem{PPH1} P. Philippon, \emph{Crit\`eres pour l'ind\'ependance alg\'ebrique}, Publ. Math. Inst. Hautes \'Etudes Sci. {\bf 64} 
(1986),  5--52. 

\bibitem{Ra92} B. Rand\'e, \'Equations fonctionnelles de Mahler et applications aux suites $p$-r\'eguli\`eres, Th\`ese de doctorat, Universit\'e de Bordeaux I, Talence, 1992.

\bibitem{Sh_Liv} A. B. Shidlovskii, Transcendental numbers, \emph{De Gruyter Studies in Mathematics} {\bf 12}, Walter de Gruyter \& Co., Berlin, 1989. 

\bibitem{Miw} M. Waldschmidt, Diophantine approximation on linear algebraic groups. Transcendence properties of the exponential function in several variables, \emph{Grundlehren der Mathematischen Wissenschaften}   
 {\bf 326}, Springer-Verlag, Berlin, 2000.
 
 \bibitem{ZS}  O. Zariski and P. Samuel, Commutative algebra II, \emph{Graduate Texts in Math.}  {\bf 29} Springer-Verlag, New York-Heidelberg, 1975.

\end{thebibliography}
\end{document}